\documentclass[12pt]{amsart}

\usepackage[hmargin=0.8in,height=8.6in]{geometry}
\usepackage{amssymb,amsthm, times}
\usepackage{delarray,verbatim}
\usepackage{ifpdf}
\ifpdf
\usepackage[pdftex]{graphicx}
 \else
\usepackage[dvips]{graphicx}
 \fi

\usepackage{bm}

\usepackage{ifpdf}
\usepackage{color,soul}
\usepackage[bordercolor=white,backgroundcolor=gray!30,linecolor=black,colorinlistoftodos]{todonotes}

\usepackage{xcolor}
\usepackage{mdframed}
{\begin{mdframed}[backgroundcolor=yellow]\begin{remark}}%
		{\end{remark}\end{mdframed}}
\definecolor{webgreen}{rgb}{0,.5,0}
\definecolor{webbrown}{rgb}{.8,0,0}
\definecolor{emphcolor}{rgb}{0.95,0.95,0.95}
\usepackage{hyperref}
\hypersetup{%
          colorlinks=true,
          linkcolor=webbrown,
          filecolor=webbrown,
          citecolor=webgreen,
          breaklinks=true}
\ifpdf \hypersetup{pdftex,
            pdfstartview=FitH, 
            bookmarksopen=true,
            bookmarksnumbered=true
} \else \hypersetup{dvips} \fi

\newcommand {\ud}{{\rm d}}

\usepackage{color}

\newcommand {\B}{\mathcal{B}}

\numberwithin{equation}{section}

\newtheorem{theorem}{Theorem}[section]
\newtheorem{proposition}{Proposition}[section]
\newtheorem{corollary}{Corollary}[section]
\newtheorem{remark}{Remark}[section]
\newtheorem{lemma}{Lemma}[section]

\numberwithin{remark}{section} \numberwithin{proposition}{section}
\numberwithin{corollary}{section}
\newcommand {\R}{\mathbb{R}}

\newcommand {\p}{\mathbb{P}}

\newcommand {\E}{\mathbb{E}}

\newcommand{\diff}{{\rm d}}

\newcommand{\lev}{L\'{e}vy }

\newcommand{\e}{\mathbb{E}}

\title[Refraction-reflection strategies in the dual model]{Refraction-reflection strategies in the dual model}
\thanks{This version: \today.  J. L. P\'erez  is  supported  by  CONACYT,  project  no.\ 241195.
K. Yamazaki is in part supported by MEXT KAKENHI Grant Number  26800092.}
\author[J. L. P\'erez]{Jos\'e-Luis P\'erez$^*$}
\thanks{$*$\, Department of Probability and Statistics, Centro de Investigaci\'on en Matem\'aticas A.C. Calle Jalisco s/n. C.P. 36240, Guanajuato, Mexico. Email: jluis.garmendia@cimat.mx.  }
\author[K. Yamazaki]{Kazutoshi Yamazaki$^\dag$}
\thanks{$\dag$\, (Corresponding author) Department of Mathematics,
Faculty of Engineering Science, Kansai University, 3-3-35 Yamate-cho, Suita-shi, Osaka 564-8680, Japan. Email: kyamazak@kansai-u.ac.jp.  Phone: +81-(0)6-6368-1527.  }
\date{}
\begin{document}

\begin{abstract}
We study the dual model with capital injection under the additional condition that the dividend strategy is absolutely continuous.  We consider a \emph{refraction-reflection strategy} that pays dividends at the maximal rate whenever the surplus is above a certain threshold, while capital is injected so that it stays positive. The resulting controlled surplus process becomes the spectrally positive version of the refracted-reflected process recently studied by P\'erez and Yamazaki \cite{YP20015}.  We study various fluctuation identities of this process and prove the optimality of the refraction--reflection strategy. Numerical results on the optimal dividend problem are also given.

\noindent \small{\textbf{Key words:}  
dividends; capital injection;  refracted-reflected \lev processes; scale functions; dual model.
}\\
\end{abstract}

\maketitle

\section{Introduction}

We revisit the dual model where the surplus of a company is modeled by a \lev process with positive jumps. This is an appropriate model for a company driven by inventions or discoveries. The seminal papers by Avanzi et al.\ \cite{AGS2007, AGS2008} studied the case of a compound Poisson process with hyperexponentially distributed jumps (with or without Brownian motion); they showed that the expected net present value (NPV) of total discounted dividends until ruin is maximized by a barrier strategy. The underlying process is then generalized by Bayraktar et al.\ \cite{BKY} to a general spectrally positive \lev process via fluctuation theory and scale functions. There are several variants of this model; see, e.g., Bayraktar et al.\ \cite{BKY2} for the case with fixed costs and Avanzi et al.\  \cite{ATW2014} and P\'erez and Yamazaki \cite{YP2016} for periodic payment opportunities.

In this paper, we introduce simultaneously two existing extensions: bail-out with capital injection and an absolutely continuous condition. 

In the former, it is assumed that the shareholders are required to provide capital injection in order to avoid ruin. For the spectrally negative \lev model, Avram et al.\ \cite{APP2007} showed that it is optimal to reflect the surplus process at zero and at some upper boundary, with the optimally controlled process being the doubly reflected \lev process of \cite{P2003}.  In the dual model, it is again optimal to reflect at two boundaries; in particular, the hyperexponential jump size case has been solved by Avanzi et al.\ \cite{ASW2011} and the spectrally positive case by Bayraktar et al.\ \cite{BKY}.

In the latter, the rate at which the dividends are paid is bounded and, instead,  strategies must be absolutely continuous with respect to the Lebesgue measure.  For a spectrally negative L\'evy surplus process, Kyprianou et al.\ \cite{KLP} showed the optimality of the \emph{refraction strategy} under a completely monotone assumption on the \lev measure; namely, it is optimal to pay dividends at the maximal rate as long as the surplus is above some fixed level. The optimally controlled process becomes then the \emph{refracted \lev process} of Kyprianou and Loeffen \cite{KL}.  The related dual model has been solved by Yin et al.\ \cite{YSP} where they showed the optimality of the refraction strategy for a general spectrally positive L\'evy process. 

Under both capital injection and the absolutely continuous assumption, it is a natural conjecture that a \emph{refraction-reflection strategy} is optimal.  Namely, it is optimal to pay dividends at the maximal rate above a certain level while injecting capital so as to stay nonnegative: the controlled process becomes a \emph{refracted-reflected} process, recently studied by P\'erez and Yamazaki \cite{YP20015}.  In this paper, we focus on the dual model and hence we need its spectrally positive version.

The objective of the paper is twofold.
\begin{enumerate}
\item We obtain fluctuation identities that will be useful in analyzing the performance of a refraction-reflection strategy.  In particular, we compute using the scale function the resolvent measure, the expected NPVs of dividend payouts and capital injection as well as the occupation time, of the controlled surplus process under refraction-reflection strategies.  We take similar steps as in the spectrally negative case \cite{YP20015}.  However, there are major differences and difficulty in pursuing these. 

Differently from the spectrally negative case, the derived expressions will contain the derivative of the scale function, which is not necessarily differentiable; one needs to be careful about the selection of the right- and left-hand derivatives that differ for the case of bounded variation when the \lev measure has atoms.  In addition, the refracted-reflected process can stay at zero for a positive amount of time and hence the derived expressions can have extra terms.
%
\item We then use these results to obtain the optimal strategy in the optimal dividend problem with capital injection described above: the optimal refraction level as well as the value function are concisely expressed in terms of the scale function. The candidate strategy is first chosen so that the value function becomes continuously differentiable at the boundary for the case of bounded variation, and twice continuously differentiable for the case of unbounded variation. Its optimality is confirmed by a verification lemma, which is adapted from related results under absolutely continuous assumptions as in e.g., Hern\'andez-Hern\'andez et al.\ \cite{HPY2015} and Kyprianou et al.\ \cite{KLP}.
\end{enumerate}

The rest of the paper is organized as follows. In Section \ref{section_preliminary}, we present an overview on scale functions and some fluctuation identities related to the spectrally positive \lev processes and their respective reflected and refracted processes.   In Section  \ref{section_refracted_reflected}, we give a construction of the refracted-reflected spectrally positive \lev process and obtain several fluctuation identities associated with this class of processes.  In Section \ref{section_dividend}, we solve the dividend problem with capital injection under the absolutely continuous assumption. Finally, in Section \ref{numerical_section}, we give numerical results on the dividend problem and confirm the optimality of the refraction-reflection strategy, along with sensitivity analyses.


\section{Preliminaries} \label{section_preliminary}

In this section, we review the spectrally positive \lev process and their reflected and refracted processes, as well as their fluctuation identities written in terms of the scale function.  Regarding the refracted spectrally positive \lev process, we begin with the observation that it can be written as the negative of the spectrally negative case as in Kyprianou and Loeffen \cite{KL}: the fluctuation identities can thus be obtained directly from \cite{KL}. 

\subsection{Spectrally positive \lev processes}
Let $Y=(Y_t; t\geq 0)$ be a L\'evy process defined on a  probability space $(\Omega, \mathcal{F}, \p)$.  For $x\in \R$, we denote by $\p_x$ the law of $Y$ when it starts at $x$ and write for convenience  $\p$ in place of $\p_0$. Accordingly, we shall write $\e_x$ and $\e$ for the associated expectation operators. In this paper, we shall assume throughout that $Y$ is \textit{spectrally positive},   meaning here that it has no negative jumps and that it is not a subordinator. Its Laplace exponent $\psi_Y:[0,\infty) \to \R$, i.e.
\[
\e\big[{\rm e}^{-\theta Y_t}\big]=:{\rm e}^{\psi_Y(\theta)t}, \qquad t, \theta\ge 0,
\]
is given by the \emph{L\'evy-Khintchine formula}
\begin{equation}\label{lk}
\psi_Y(\theta):=\gamma_Y\theta+\frac{\sigma^2}{2}\theta^2+\int_{(0, \infty)}\big({\rm e}^{-\theta x}-1+\theta x\mathbf{1}_{\{x<1\}}\big)\Pi(\ud x), \quad \theta \geq 0,
\end{equation}
where $\gamma_Y \in \R$, $\sigma\ge 0$, and $\Pi$ is a measure on $(0, \infty)$ called the L\'evy measure of $Y$ that satisfies
\[
\int_{(0, \infty)}(1\land x^2)\Pi(\ud x)<\infty.
\]

It is well-known that $Y$ has paths of bounded variation if and only if $\sigma=0$ and $\int_{(0,1)} x\Pi(\mathrm{d}x) < \infty$; in this case, $Y$ can be written as
\begin{equation}
Y_t=-c_Yt+S_t, \,\,\qquad t\geq 0,\notag
\end{equation}
where 
\begin{align}
c_Y:=\gamma_Y+\int_{(0,1)} x\Pi(\mathrm{d}x) \label{def_drift_finite_var}
\end{align}
 and $(S_t; t\geq0)$ is a driftless subordinator. Note that  necessarily $c_Y>0$, since we have ruled out the case that $Y$ has monotone paths; its Laplace exponent is given by
\begin{equation*}
\psi_Y(\theta) = c_Y \theta+\int_{(0,\infty)}\big( {\rm e}^{-\theta x}-1\big)\Pi(\ud x), \quad \theta \geq 0.
\end{equation*}

The \emph{\lev process reflected at the lower boundary $0$} is a strong Markov process written concisely by
 \begin{align}
 	U_t := Y_t+\sup_{0 \leq s\leq t}(-Y_s)\vee0,\qquad t\geq0. \label{reflected_levy}
 \end{align}
 The supremum term pushes the process upward whenever it attempts to down-cross the level $0$; as a result the process only takes values on $[0, \infty)$.

For fixed $\delta \geq 0$ and $b \in \R$, the \emph{refracted spectrally positive \lev process $A$} is defined as the unique strong solution to the stochastic differential equation (SDE)
\begin{equation}
 A_t= Y_t- \delta \int_0^t 1_{\{A_s>b\}}\diff s,\qquad\text{$t\geq0$.}\label{refracted_levy}
 \end{equation}
 Informally speaking, a linear drift at rate $\delta$ is subtracted from the increments of  the underlying \lev process $Y$ whenever it exceeds $b$. 
 
The existence and uniqueness of this process is immediate by the observation that its dual process $-A$ becomes a refracted spectrally negative \lev process of \cite{KL}.  Indeed, if we define a drift-changed process
 \begin{align}
 X_t := Y_t - \delta t, \quad t \geq 0,\label{def_X}
 \end{align}
 the above SDE is equivalent to 
 \begin{align}
 -A_t= -Y_t+ \delta \int_0^t 1_{\{A_s>b\}}\diff s =  - X_t - \delta \int_0^t  1_{\{ A_s \leq b\}}\diff s = -X_t - \delta \int_0^t 1_{\{-A_s \geq -b\}}\diff s. \label{connection_w_SN_case}
 \end{align}
Given that $\p_x (-A_t = -b ) = 0$ for Lebesgue a.e.\ $t > 0$ by Corollary 22 of \cite{KL}, this is the  SDE describing the refracted spectrally negative \lev process $-A$ with the underlying spectrally negative \lev process $-X$ and the refraction level $-b$.

\subsection{Review of scale functions}  \label{section_scale_functions}
Before discussing further on the refracted \lev process $A$, we review here the scale function and its applications on the spectrally positive \lev process and its reflected process. As we need to deal with the fluctuation of the two processes $Y$ and $X$  to describe those of their associated refracted-reflected processes, we define two scale functions here.

Fix $q \geq 0$. We use $\mathbb{W}^{(q)}$ and $W^{(q)}$ for the scale functions of the spectrally negative \lev processes $-Y$ and $-X$, respectively.  These are the mappings from $\R$ to $[0, \infty)$ that take value zero on the negative half-line, while on the positive half-line they are strictly increasing functions that are defined by their Laplace transforms:
\begin{align} \label{scale_function_laplace}
\begin{split}
\int_0^\infty  e^{-\theta x} \mathbb{W}^{(q)}(x) \diff x &= \frac 1 {\psi_Y(\theta) -q}, \quad \theta > \varphi(q), \\
\int_0^\infty  e^{-\theta x} W^{(q)}(x) \diff x &= \frac 1 {\psi_X(\theta)-q}, \quad \theta > \Phi(q),
\end{split}
\end{align}
where $\psi_X(\theta) := \psi_Y(\theta) + \delta \theta$, $\theta \geq 0$, is the Laplace exponent for $X$ and
\begin{align}
\begin{split}
\varphi(q) := \sup \{ \lambda \geq 0: \psi_Y(\lambda) = q\}  \quad \textrm{and} \quad  \Phi(q) := \sup \{ \lambda \geq 0: \psi_X(\lambda) = q\} . 
\end{split}
\label{def_varphi}
\end{align}
In particular, when $q=0$, we shall drop the superscript.
By the strict  convexity of $\psi_Y$ on $[0, \infty)$, we derive the inequality $\varphi(q) > \Phi(q) > 0$ for $q > 0$ and  $\varphi(q) \geq \Phi(q) \geq 0$ for $q = 0$.

We also define, for $x \in \R$, 
\begin{align*}
\overline{\mathbb{W}}^{(q)}(x) :=  \int_0^x \mathbb{W}^{(q)}(y) \diff y \quad \textrm{and} \quad
\mathbb{Z}^{(q)}(x) := 1 + q \overline{\mathbb{W}}^{(q)}(x).
\end{align*}
Noting that $\mathbb{W}^{(q)}(x) = 0$ for $-\infty < x < 0$, we have

\begin{align}
\overline{\mathbb{W}}^{(q)}(x) = 0 \quad \textrm{and} \quad \mathbb{Z}^{(q)}(x) = 1, \quad x \leq 0.
\end{align}
In addition, we define $\overline{W}^{(q)}$ and $Z^{(q)}$ analogously for $-X$.  These scale functions are related by the following equalities 
%
\begin{align}\label{RLqp}
&\delta \int_0^x\mathbb{W}^{(q)}(x-y)  W^{(q)}(y) \ud y=\overline{\mathbb{W}}^{(q)}(x)-\overline{W}^{(q)}(x), \quad  x \in \R \; \textrm{and} \; q \geq 0,
\end{align}
which can be proven by showing that the Laplace transforms on both sides are equal.

Regarding their asymptotic values as $x \downarrow 0$ we have, as in Lemmas 3.1 and 3.2 of \cite{KKR}, 
\begin{align}\label{eq:Wqp0}
\begin{split}
\mathbb{W}(0)=  \mathbb{W}^{(q)} (0) &= \left\{ \begin{array}{ll} 0 & \textrm{if $Y$ is of unbounded
variation,} \\ c_Y^{-1} & \textrm{if $Y$ is of bounded variation,}
\end{array} \right. \\
W(0)=  W^{(q)} (0) &= \left\{ \begin{array}{ll} 0 & \textrm{if $X$ is of unbounded
variation,} \\ c_X^{-1} & \textrm{if $X$ is of bounded variation,}
\end{array} \right.  
\end{split}
\end{align}
and 
\begin{align} \label{W_zero_derivative}
\begin{split}
\mathbb{W}^{(q) \prime}_+ (0) &:= \lim_{x \downarrow 0}\mathbb{W}^{(q)\prime}_+ (x) =
\left\{ \begin{array}{ll}  \frac 2 {\sigma^2} & \textrm{if }\sigma > 0, \\
\infty & \textrm{if }\sigma = 0 \; \textrm{and} \; \Pi(0,\infty) = \infty, \\
\frac {q + \Pi(0, \infty)} {c_Y^2} &  \textrm{if }\sigma = 0 \; \textrm{and} \; \Pi(0,\infty) < \infty,
\end{array} \right. \\
W^{(q)\prime}_+ (0) &:= \lim_{x \downarrow 0}W^{(q)'}_+ (x) =
\left\{ \begin{array}{ll}  \frac 2 {\sigma^2} & \textrm{if }\sigma > 0, \\
\infty & \textrm{if }\sigma = 0 \; \textrm{and} \; \Pi(0, \infty) = \infty, \\
\frac {q + \Pi(0, \infty)} {c_X^2} &  \textrm{if }\sigma = 0 \; \textrm{and} \; \Pi(0, \infty) < \infty,
\end{array} \right. 
\end{split}
\end{align}
and, as in Lemma 3.3 of \cite{KKR}, 
\begin{align}
\begin{split}
 e^{-\varphi(q) x}\mathbb{W}^{(q)} (x) \nearrow \psi_Y'(\varphi(q))^{-1} \quad \textrm{and} \quad  e^{-\Phi(q) x}W^{(q)} (x) \nearrow \psi_X'(\Phi(q))^{-1} \quad \textrm{as } x \rightarrow \infty,
\end{split}
\label{W_q_limit}
\end{align}
where in the case  $\psi_{Y,+}'(0) = 0$  or $\psi_{X,+}'(0) = 0$, the right-hand side, when $q=0$,  is understood to be infinity. Here and for the rest of the paper, $g_+'(x)$ and $g_-'(x)$,  for any function $g$, are the right-hand and left-hand derivatives, respectively, at $x$.

\begin{remark} \label{remark_smoothness} 
	It is known that the right-hand and left-hand derivatives of the scale function always exist for all $x > 0$.
If $Y$ is of unbounded variation or the \lev measure is atomless, it is known that $\mathbb{W}^{(q)}$ and $W^{(q)}$ are $C^1(\R \backslash \{0\})$.  For more comprehensive results on the smoothness, see \cite{Chan2011}.
\end{remark}

We shall see in later sections that the paths of a refracted-reflected \lev process can be decomposed into those of $X$ and $U$, which are defined in \eqref{def_X} and \eqref{reflected_levy}, respectively.  Here, we summarize a few known identities of these processes in terms of the scale functions that will be used later in the paper.

For the drift-changed process {$X$, let us define the first down- and up-crossing times, respectively, by
\begin{align}
	\label{first_passage_time}
	\tau_a^- := \inf \left\{ t > 0: X_t < a \right\} \quad \textrm{and} \quad \tau_a^+ := \inf \left\{ t > 0: X_t >  a \right\}, \quad a \in \R;
\end{align}
here and throughout, let $\inf \varnothing = \infty$.
Then, for any $a > b$ and $x \leq a$, 
\begin{align}
	\begin{split}
		\E_x \left( e^{-q \tau_b^-} 1_{\left\{ \tau_a^+ > \tau_b^- \right\}}\right)  &= \frac {W^{(q)}(a-x)}  {W^{(q)}(a-b)}, \\
		\E_x \left( e^{-q \tau_a^+} 1_{\left\{ \tau_a^+ < \tau_b^- \right\}}\right)&= Z^{(q)}(a-x) -  Z^{(q)}(a-b) \frac {W^{(q)}(a-x)}  {W^{(q)}(a-b)}.
	\end{split}
	\label{laplace_in_terms_of_z}
\end{align}
In addition, the \emph{$q$-resolvent measure} is known, for any Borel set $B$ on $[b,a]$, to have the following form 
\begin{align} \label{resolvent_density}
\begin{split}
	\E_x \Big( \int_0^{ \tau^+_a \wedge \tau_{b}^- } e^{-qt} 1_{\left\{ X_t \in B  \right\}} \diff t\Big) 
	=\frac{W^{(q)}(a-x)}{W^{(q)}(a-b)} \overline{\Gamma}^{(q)}_a(a-b;B)-\overline{\Gamma}^{(q)}_a(a-x;B),  \quad b \leq x \leq a,
	\end{split}
\end{align} 
where  
\begin{align*}
\overline{\Gamma}^{(q)}_a(l;B):= \int_{a-l}^a 1_{\{y\in B\}} W^{(q)}(y-a+l)\diff y;
\end{align*}
 see Theorem 8.7 of \cite{K}. 


Regarding the reflected process $U$, let
\begin{align}
	\eta^+_b:=\inf\{t>0:U_t>b\}, \quad b \in \R. \label{def_kappa_time}
\end{align}
By Theorem 1 (ii) of \cite{P2004}, 
\begin{align}\label{rsr}
	\e_x\bigg(\int_0^{\eta_b^+}e^{-qt}1_{\{ U_t \in B \}}\ud t \bigg)=\frac{\mathbb{W}^{(q)}(b-x)}{\mathbb{W}_+^{(q)\prime}(b)}\Gamma_b^{(q) \prime}(b;B)-\Gamma_b^{(q)}(b-x;B), \quad 0 \leq  x \leq b,
\end{align} 
where 
\begin{align*}
\Gamma_b^{(q)}(l;B)&:=  \int_{b-l}^b1_{\{y\in B\}}\mathbb{W}^{(q)}(y-b+l)\diff y,\\
\Gamma_b^{(q)\prime}(l;B)&:=  1_{\{ b-l\in B\}} \mathbb{W}(0) + \int_{b-l}^b1_{\{y\in B\}}\mathbb{W}^{(q) \prime}_+(y-b+l)\diff y.
\end{align*}
In particular,
\begin{align} \label{upcrossing_time_reflected}
	\e_x\big(e^{-q\eta_b^+}\big)=\mathbb{Z}^{(q)}(b-x)  -q\frac{\mathbb{W}^{(q)}(b-x)}{\mathbb{W}^{(q)\prime}_+(b)} \mathbb{W}^{(q)}(b), \quad 0 \leq x \leq b.
\end{align}
\par
It is known that a spectrally positive \lev process creeps upwards (i.e.\ $\p_x ( Y_{\tilde{\tau}_b^+}= b, \tilde{\tau}_b^+ < \infty ) > 0$ for $x < b$ with $\tilde{\tau}_b^+$ being the first up-crossing time for $Y$) if and only if $\sigma > 0$ (see Exercise 7.6 of \cite{K}). By this and Theorem 1 in \cite{P2007}, if $Y$ is of bounded variation,  the joint law of the first up-crossing time and the overshoot at $b$ for $U$ is given by, for any bounded measurable function $ h :[0,\infty)\to\R$, 
\begin{align}\label{overshoot_reflected}
\begin{split}
	\e_x\bigg(e^{-q\eta_b^+}h(U_{\eta_b^+})\bigg)&=\int_0^{b}\int_{(b-y, \infty)}h(y+u)\left\{\frac{\mathbb{W}^{(q)}(b-x)}{\mathbb{W}_+^{(q)\prime}(b)}\mathbb{W}_+^{(q)\prime}(y)-\mathbb{W}^{(q)}(y-x)\right\}\Pi(\diff u)\diff y \\
	&\qquad +\mathbb{W}(0)\int_{(b, \infty)} h (u) \frac{\mathbb{W}^{(q)}(b-x)}{\mathbb{W}_+^{(q)\prime}(b)} \Pi(\diff u).
	\end{split}
\end{align}
In addition, if we define, for $t \geq 0$, $\widetilde{R}_t := \sup_{s\leq t}(-Y_s)\vee0$ so that $U_t = Y_t + \widetilde{R}_t$, then
\begin{align}
	\mathbb{E}_x\Big(\int_{[0,\eta_b^+]}e^{-qt}\ud \widetilde{R}_t\Big)&= \frac{\mathbb{W}^{(q)}(b-x)}{\mathbb{W}_+^{(q)\prime}(b)}, \quad 0 \leq x \leq b; \label{capital_injection_identity_SP}
\end{align}
see page 167 in the proof of Theorem 1 of \cite{APP2007}.



\subsection{Fluctuations of refracted spectrally positive \lev processes}

We now discuss how the fluctuation identities of the refracted spectrally positive \lev process $A$ can be computed using directly the results on the spectrally negative case by Kyprianou and Loeffen \cite{KL}.
 


Let us note that if we consider the process $(a-A_t;  t\geq0)$ then it satisfies, by \eqref{connection_w_SN_case},
\begin{align*}
a-A_t=a-X_t-\delta\int_0^t1_{\{-A_s \geq -b\}}\diff s=a-X_t-\delta\int_0^t1_{\{a-A_s \geq a-b\}}\diff s, \quad t \geq 0,
\end{align*}
see the discussion after \eqref{connection_w_SN_case} on how the event $\{a-A_s = a-b\}$ can be ignored.
Therefore $a-A$ is a refracted spectrally negative L\'evy process, with the driving process $-X$ starting at $a-x$ with the refraction level $a-b$. On the other hand, we note that
\begin{align*}
&\inf\{t>0:A_t<0\}=\inf\{t>0:a-A_t>a\}, \quad\text{and}\\
&\inf\{t>0:A_t>a\}=\inf\{t>0:a-A_t<0\}.
\end{align*}}

Using these observations, the spectrally positive versions of the results in \cite{KL} (in particular, Theorems 4, 5, and 6) can be derived simply by change of variables.  In addition, as their corollary, these give the expected NPV of dividends until ruin under the refraction strategy in the dual model as in Yin et al.\ \cite{YSP}.

\section{Refracted-reflected spectrally positive L\'evy processes}\label{section_refracted_reflected}

\par For fixed $b > 0$, we construct the path of a \lev process which is reflected at the lower barrier $0$ and refracted at the upper barrier $b$. The stochastic process moves on $[0,b)$ as a reflected \lev process. Whenever the process is above $b$, a linear drift at rate $\delta$ is subtracted from the increments of the reflected \lev process.  
The process can be formally constructed by the recursive algorithm given below; while it is essentially the same as in the spectrally negative case \cite{YP20015},  we provide it here for the sake of completeness.

\begin{center}
	\line(1,0){300}
\end{center}
\textbf{Construction of the refracted-reflected spectrally positive \lev process $V$ under $\p_x$}
\begin{description}
	\item[Step 0] Set $V_{0-}=x$.  If $x \geq 0$, then set $\underline{\tau} = 0$ and go to \textbf{Step 1}.  Otherwise, set $\overline{\tau} = 0$ and go to \textbf{Step 2}.
	\item[Step 1] Let $( \widetilde{A}_t; t \geq \underline{\tau} )$ be the refracted \lev process (with refraction level $b$) that starts at the time $\underline{\tau}$ at the level $x$, and $\overline{\tau} := \inf \{ t > \underline{\tau}: \widetilde{A}_t < 0\}$.
	Set $V_t=\widetilde{A}_t$ for all $\underline{\tau} \leq t < \overline{\tau}$. Then go to \textbf{Step 2}.
	\item[Step 2] Let $( \widetilde{U}_t; t \geq \overline{\tau})$ be the \lev process reflected at the lower boundary $0$ that starts at the time $\overline{\tau}$ at the level $0$, and $\underline{\tau} := \inf \{ t > \overline{\tau}: \widetilde{U}_t > b\}$. Set $V_t=\widetilde{U}_t$ for all $\overline{\tau} \leq t < \underline{\tau}$ and $x = \widetilde{U}_{\underline{\tau}}$. 
		Then go to  \textbf{Step 1}.\end{description}
\begin{center}
	\line(1,0){300}
\end{center}

 In view of the construction above, $V$ admits a decomposition
\begin{align*}
	V_t = Y_t  + R_t- L_t, \quad t \geq 0,
\end{align*}
where both $R$ and $L$ are nondecreasing and right-continuous processes such that $R_{0-} = L_{0-} = 0$.  The former pushes the process upward when it attempts to go below $0$ and the latter pulls it downward  when it is above $b$ and can be written
\begin{align}
	L_t = \delta \int_0^t 1_{\{ V_s > b \}} \ud s, \quad t \geq 0. \label{L_def}
\end{align}
In the dual model with capital injection, $R_t$ models the cumulative amount of injected capital until $t$ while $L_t$ is that of dividends.
%
A difference from the spectrally negative case is that the process $R$ is continuous for $0 < t < \infty$ with a jump at $t = 0$ when $x < 0$.

Our derivation of the results of this section relies on the following remark on its connection with the drift-changed process $X$ and the reflected process $U$. Let us denote, for $a > 0$
 \begin{align}
 T_a^+:=\inf\{t>0:V_t>a\}\quad  
\textrm{and} \quad T_a^-:=\inf\{t>0:V_t<a\}. \label{def_T_hitting}
\end{align}

\begin{remark} \label{remark_connection_Y_U} Recall the hitting times $\tau_b^-$ of $X$  and $\eta_b^+$ of $U$ as in \eqref{first_passage_time} and \eqref{def_kappa_time}, respectively.  For any $x \in \R$, $\p_x$-a.s, we have $T_b^+ =\eta_b^+$ and $V_t = U_t$ on $0 \leq t \leq T_b^+$; similarly, we have $T_b^- =\tau_b^-$ and $V_t = X_t$  on $0 \leq t \leq T_b^-$.
\end{remark}

Using this and the strong Markov property, we can apply the same technique as in, for example, \cite{KL} and \cite{YP20015}. 
Fluctuation identities are first obtained for the case of bounded variation using the fact that $\mathbb{W}(0)$ is strictly positive (see \eqref{eq:Wqp0}). This is then extended to the case of unbounded variation by the following proposition and remark. Recall that, as in Definition 11 of \cite{KL},
 a sequence of processes $\{(\xi_s^{(n)})_{s \geq 0};n\geq1\}$ is \emph{strongly approximating} for a process $\xi$, if $\lim_{n \uparrow \infty}\sup_{0 \leq s \leq t} |\xi_s - \xi^{(n)}_{s}| =0$ for any $t > 0$ a.s.  In addition, for any spectrally positive \lev process $Y$, there exists a strongly approximating sequence $Y^{(n)}$ of spectrally positive \lev processes with paths of bounded variation (see page 210 of \cite{B}). The proof of the proposition below is essentially the same as in the spectrally negative case \cite{YP20015} and is hence omitted.

%

\begin{proposition} \label{prop_approximation}
	Suppose $Y$ is of unbounded variation and $(Y^{(n)}; n \geq 1)$ is a strongly approximating sequence for $Y$. In addition, define $V$ and  $V^{(n)}$ as the refracted-reflected processes associated with  $Y$ and $Y^{(n)}$, respectively. Then $V^{(n)}$ is a strongly approximating sequence of $V$. 
\end{proposition}

\begin{remark} \label{remark_strongly_approximating} 
	 \label{remark_continuity_theorem_scale_function}Suppose $(Y^{(n)}; n \geq 1)$ is a strongly approximating sequence for $Y$ and $(\mathbb{W}^{(q)}_n; n \geq 1)$ and $(W^{(q)}_n; n \geq 1)$ are the corresponding scale functions of $-Y^{(n)}$ and  $-X^{(n)} := -( Y^{(n)}_t - \delta t; t \geq 0)$ respectively. 
	 
As is discussed in Lemma 20 of \cite{KL}, by the fact that the Laplace transform of the measure $\mathbb{W}^{(q)} (\diff x) = \mathbb{W}^{(q)} (0) \delta_0(\diff x) + \mathbb{W}^{(q) \prime} (x) \diff x$ (for any scale function $\mathbb{W}^{(q)}$) can be written in terms of the Laplace exponent, the continuity theorem implies
 $\mathbb{W}^{(q)}_n(x)$ (resp.\  $W^{(q)}_n(x)$) converges to $\mathbb{W}^{(q)}(x)$ (resp.\ $W^{(q)}(x)$) for every $x \in \R$. 
 
This also shows the following convergence of the derivatives: if $Y$ is of unbounded variation,  for $x > 0$, the right-hand derivatives $\mathbb{W}_{n,+}^{(q)\prime}(x)$ and $W_{n,+}^{(q)\prime}(x)$ converge to $\mathbb{W}^{(q) \prime}(x)$ and $W^{(q) \prime}(x)$, respectively.  To see this, by \eqref{upcrossing_time_reflected},
\begin{align} 
	\e \big(e^{-q\eta_b^+}\big)= 1 + q \int_0^b \mathbb{W}^{(q)}(y) \diff y  -q\frac{(\mathbb{W}^{(q)}(b))^2}{\mathbb{W}^{(q)\prime}_+(b)}, \quad b > 0, \label{ratio_scale_function_identity}
\end{align}
and the same identity holds when $\eta_b^+$ is replaced with $\eta_{b,n}^+$ (which is the upcrossing time \eqref{def_kappa_time} for the reflected process $U^{(n)}$ of $X^{(n)}$), $\mathbb{W}^{(q)}$ with $\mathbb{W}_{n}^{(q)}$, and $\mathbb{W}^{(q)\prime}_+$ with $\mathbb{W}_{n,+}^{(q)\prime}$.  Because $(U^{(n)}; n \geq 1)$ is also strongly approximating for $U$ (see the proof of Proposition 2.1 of \cite{YP20015}), we have $\eta_{b-\delta}^+ \leq \liminf_{n \uparrow \infty} \eta_{b,n}^+ \leq \limsup_{n \uparrow \infty} \eta_{b,n}^+ \leq \eta_{b+\delta}^+$ for any $0 < \delta < b$ a.s.  Because $Y$ is of unbounded variation, thanks to the regularity of the upper half-line (see page 232 of \cite{K}), we have $\eta_b^+ = \lim_{\delta \downarrow 0}\eta_{b+\delta}^+ = \lim_{\delta \downarrow 0}\eta_{b-\delta}^+$ a.s.  Hence, $\lim_{n \uparrow \infty} \eta_{b,n}^+ = \eta_{b}^+$ a.s. Now,  dominated convergence gives $\lim_{n \uparrow \infty}\e \big(e^{-q\eta_{b,n}^+}\big) = \e \big(e^{-q\eta_b^+}\big)$; therefore,
in view of \eqref{ratio_scale_function_identity} and the fact that $\lim_{n \uparrow \infty}\mathbb{W}_n^{(q)}(y) = \mathbb{W}^{(q)}(y)$, $y > 0$, we must have $\lim_{n \uparrow \infty}\mathbb{W}_{n,+}^{(q)\prime}(b) = \mathbb{W}_+^{(q)\prime}(b)$ (which equals $\mathbb{W}^{(q)\prime}(b)$ by the continuity of the derivative for the case of unbounded variation). The convergence of $W^{(q) \prime}_+$ also holds in the same way.

\end{remark}

\subsection{Simplifying formula} Before we obtain the fluctuation identities, we provide some formulae that  will be helpful in achieving concise expressions. The items (i) and (ii) below are borrowed from Theorem 3.1 of \cite{YP20015} (see also Theorem 2 in \cite{Ro2015} and Lemma 1 of \cite{Re2014}); items (i') and (ii') are obtained by taking right-hand derivatives (see Remark \ref{remark_differentiability} below).
\begin{lemma}\label{lemma_useful_identitysp} 
	Suppose that $Y$ is  of bounded variation. 
	For any $\tilde{p},\tilde{q} \geq0$, the following holds.  
	\begin{itemize}
		\item[(i)] For $\alpha < \beta \leq \gamma$, we have
	\begin{align*}
		\begin{split}
			\int_{0}^{\gamma-\beta}\int_{(y,\infty)}&W^{(\tilde{q})}(y-u+\beta-\alpha)\mathbb{W}^{(\tilde{p})}(\gamma-\beta-y)\Pi(\diff u)\diff y \\
			&=\mathbb{W}(0)^{-1} W^{(\tilde{q})}(\beta-\alpha) \mathbb{W}^{(\tilde{p})}(\gamma-\beta)-W^{(\tilde{q})}(\gamma-\alpha) \\
			&+\int_{\beta}^{\gamma}\mathbb{W}^{(\tilde{p})}(\gamma-y)\left((\tilde{q}-\tilde{p})W^{(\tilde{q})}(y-\alpha)-\delta W_+^{(\tilde{q})\prime}(y-\alpha)\right)\diff y. \end{split}
	\end{align*}
	\item[(i')] For $\alpha < \beta < \gamma$, we have
		\begin{align*}
			\int_{0}^{\gamma-\beta} \int_{(y,\infty)} &W^{(\tilde{q})}(y -u+\beta-\alpha)\mathbb{W}_+^{(\tilde{p})\prime}(\gamma-\beta-y)\Pi(\diff u)\diff y + \mathbb{W}(0) \int_{(\gamma-\beta,\infty)}W^{(\tilde{q})}(\gamma-u-\alpha)\Pi(\diff u)\notag\\
			&=\mathbb{W}(0)^{-1} W^{(\tilde{q})}(\beta-\alpha) \mathbb{W}_+^{(\tilde{p})\prime}(\gamma-\beta)-W_+^{(\tilde{q})\prime}(\gamma-\alpha) \notag\\
			&+\int_{\beta}^{\gamma}\mathbb{W}^{(\tilde{p})\prime}_+(\gamma-y)\left((\tilde{q}-\tilde{p})W^{(\tilde{q})}(y-\alpha)-\delta W^{(\tilde{q})\prime}_+(y-\alpha)\right)\diff y\notag\\
			&+\mathbb{W}(0)\Big((\tilde{q}-\tilde{p})W^{(\tilde{q})}(\gamma-\alpha)-\delta W_+^{(\tilde{q})\prime}(\gamma-\alpha) \Big),
		\end{align*}
		\item[(ii)] For $\alpha < \beta \leq \gamma$, we have
\begin{align*} 
\begin{split}
\int_{0}^{\gamma-\beta} \int_{(y, \infty)} &Z^{(\tilde{q})}(y-u+\beta-\alpha)\mathbb{W}^{(\tilde{p})}(\gamma-\beta-y)\Pi(\ud u)\ud y \\
&= \mathbb{W}(0)^{-1} Z^{(\tilde{q})}(\beta-\alpha) \mathbb{W}^{(\tilde{p})}(\gamma-\beta)-Z^{(\tilde{q})}(\gamma-\alpha)-(\tilde{p}-\tilde{q})\overline{\mathbb{W}}^{(\tilde{p})}(\gamma-\beta) \\
&+\tilde{q}\int_{\beta}^{\gamma}\mathbb{W}^{(\tilde{p})}(\gamma-y)\left((\tilde{q}-\tilde{p})\overline{W}^{(\tilde{q})}(y-\alpha)-\delta W^{(\tilde{q})}(y-\alpha)\right)\ud y. \end{split}
\end{align*}
		\item[(ii')] For $\alpha < \beta < \gamma$, we have
\begin{align*}
\int_{0}^{\gamma-\beta} \int_{(y,\infty)} &Z^{(\tilde{q})}(y - u+\beta-\alpha)\mathbb{W}_+^{(\tilde{p})\prime}(\gamma-\beta-y)\Pi(\diff u)\diff y + \mathbb{W}(0) \int_{(\gamma-\beta,\infty)} Z^{(\tilde{q})}(\gamma-u-\alpha)\Pi(\diff u)\\
&=\mathbb{W}(0)^{-1} Z^{(\tilde{q})}(\beta-\alpha) \mathbb{W}_+^{(\tilde{p})\prime}(\gamma-\beta)-\tilde{q}( \delta \mathbb{W}(0) +1)W^{(\tilde{q})}(\gamma-\alpha) \\ &+(\tilde{q}-\tilde{p})\mathbb{W}^{(\tilde{p})}(\gamma-\beta) + \tilde{q} \mathbb{W}(0) (\tilde{q}-\tilde{p}) \overline{W}^{(\tilde{q})}(\gamma-\alpha) \\
&+\tilde{q}\int_{\beta}^{\gamma}\mathbb{W}_+^{(\tilde{p})\prime}(\gamma-y)\left((\tilde{q}-\tilde{p})\overline{W}^{(\tilde{q})}(y-\alpha)-\delta W^{(\tilde{q})}(y-\alpha)\right)\diff y.\end{align*}
		\end{itemize}
	\end{lemma}
\begin{remark} \label{remark_differentiability} 
		In deriving (i') from (i) in Lemma \ref{lemma_useful_identitysp},  the right-hand derivative on the left-hand side can be interchanged over integrations by the following arguments.
For $\epsilon > 0$ and $0 < \delta < \gamma- \beta$,  define
\begin{align*}
		K_1(\delta, \epsilon)&:=\int_{0}^{\gamma-\beta-\delta}\int_{(y,\infty)}W^{(\tilde{q})}(y-u+\beta-\alpha)  \frac {\mathbb{W}^{(\tilde{p})}(\gamma+\epsilon-\beta-y) - \mathbb{W}^{(\tilde{p})}(\gamma-\beta-y) } \epsilon  \Pi(\diff u)\diff y, \\
		K_2(\delta, \epsilon)&:= \int_{\gamma-\beta-\delta}^{\gamma-\beta}\int_{(y,\infty)}W^{(\tilde{q})}(y-u+\beta-\alpha)  \frac {\mathbb{W}^{(\tilde{p})}(\gamma+\epsilon-\beta-y) - \mathbb{W}^{(\tilde{p})}(\gamma-\beta-y) } \epsilon  \Pi(\diff u)\diff y. 
	\end{align*}
Note that for all $0 < \delta < \gamma- \beta$, we have 
\begin{align} \label{K_sum}
K_1(0, \epsilon) = K_1(\delta, \epsilon) + K_2(\delta, \epsilon).
\end{align}  
Here, we show how $\lim_{\epsilon \downarrow 0}K_1(0, \epsilon)$ can be computed.
For the first term, for any $0 < \epsilon < \bar{\epsilon}$ for fixed $\bar{\epsilon} > 0$ and $0 < y < \gamma-\beta - \delta$, we have a bound: $| \mathbb{W}^{(\tilde{p})}(\gamma+\epsilon-\beta-y) - \mathbb{W}^{(\tilde{p})}(\gamma-\beta-y) | / \epsilon \leq \sup_{\delta < z < \gamma - \beta + \bar{\epsilon}} \mathbb{W}_+^{(\tilde{p}) \prime}(z) < \infty$ (because $\mathbb{W}_+^{(\tilde{p}) \prime}(z)$ is finite if $z > 0$, which is clear from \eqref{capital_injection_identity_SP} [see  also identity (8.26) in \cite{K}]), and 
	\begin{multline*}
\int_{0}^{\gamma-\beta-\delta}\int_{(y,\infty)}W^{(\tilde{q})}(y-u+\beta-\alpha) \sup_{\delta < z < \gamma - \beta + \bar{\epsilon}} \mathbb{W}_+^{(\tilde{p}) \prime}(z)  \Pi(\diff u)\diff y \\
		 \leq W^{(\tilde{q})}(\beta-\alpha) \sup_{\delta < z < \gamma - \beta + \bar{\epsilon}} \mathbb{W}_+^{(\tilde{p}) \prime}(z)  \int_{0}^{\gamma-\beta-\delta}\int_{(y,y + \beta - \alpha]} \Pi(\diff u)\diff y, 
	\end{multline*}
	which is finite because, for sufficiently small $c > 0$, by the assumption that $Y$ is of bounded variation,
	\begin{align*}	
\int_{0}^{\gamma-\beta}\int_{(y,y + \beta - \alpha]} 1_{\{u < c \}}\Pi(\diff u)\diff y \leq \int_{0}^c \int_{(y,c)} \Pi(\diff u)\diff y = \int_{(0,c)} u \Pi(\diff u) < \infty.
\end{align*}
Therefore, by dominated convergence, the limit as $\epsilon \downarrow 0$ can be interchanged over the integral and hence,
	\begin{align*}
		K:= \lim_{\delta \downarrow 0}\lim_{\epsilon \downarrow 0} K_1(\delta, \epsilon) &=\lim_{\delta \downarrow 0}\int_{0}^{\gamma-\beta-\delta}\int_{(y,\infty)}W^{(\tilde{q})}(y-u+\beta-\alpha)  \mathbb{W}_+^{(\tilde{p}) \prime}(\gamma-\beta-y) \Pi(\diff u)\diff y\\
		&=\int_{0}^{\gamma-\beta}\int_{(y,\infty)}W^{(\tilde{q})}(y-u+\beta-\alpha)  \mathbb{W}_+^{(\tilde{p}) \prime}(\gamma-\beta-y) \Pi(\diff u)\diff y.
	\end{align*}
On the other hand, by Fubini's theorem,
	\begin{align*}
	0 \leq K_2(\delta,\epsilon) & \leq W^{(\tilde{q})}(\beta-\alpha)\Pi(\gamma-\beta-\delta,\gamma-\alpha) \frac 1 \epsilon \int_{\gamma-\beta-\delta}^{\gamma-\beta} \int_0^\epsilon \mathbb{W}^{(\tilde{p})\prime}_+(\gamma+z-\beta-y) \diff z \diff y\\
	&= W^{(\tilde{q})}(\beta-\alpha)\Pi(\gamma-\beta-\delta,\gamma-\alpha)\frac{1}{\epsilon}\int_0^{\epsilon}( \mathbb{W}^{(\tilde{p})}(\delta+z) - \mathbb{W}^{(\tilde{p})}(z)) \diff z\\
	&\leq W^{(\tilde{q})}(\beta-\alpha)\Pi(\gamma-\beta-\delta,\gamma-\alpha)\sup_{0\leq z\leq \epsilon}|\mathbb{W}^{(\tilde{p})}(\delta+z) - \mathbb{W}^{(\tilde{p})}(z)|.
	\end{align*}
	Hence, noting that $\lim_{\delta \downarrow 0} \lim_{\epsilon \downarrow 0}K_2(\delta,\epsilon)=0$ and by \eqref{K_sum},
	\begin{multline*}
K  = \lim_{\delta \downarrow 0}\liminf_{\epsilon \downarrow 0} (K_1(\delta, \epsilon) + K_2 (\delta, \epsilon)) = \liminf_{\epsilon\downarrow 0} K_1(0, \epsilon)  \leq  \limsup_{\epsilon\downarrow 0} K_1(0, \epsilon) \\ = 
	\lim_{\delta \downarrow 0}\limsup_{\epsilon\downarrow 0} (K_1(\delta, \epsilon) + K_2 (\delta, \epsilon)) = K,
	\end{multline*}
implying $K_1(0, \epsilon) \xrightarrow{\epsilon \downarrow 0} K$, as desired.
The same technique can be used to derive (ii') from (ii).
\end{remark}


Fix $b > 0$. We define, for $l \in \R$ and $q \geq 0$, 
  \begin{align}
  \begin{split}
  r^{(q)}(l;a)&:=W^{(q)}(l)+\delta\int_{a-b}^{l}\mathbb{W}^{(q)}(l-z)W^{(q)\prime}_+(z) \diff z, \quad a > b, \\
\hat{r}^{(q)}(l) &:= e^{-\Phi(q)(b-l)}+\delta\Phi(q)\int_{b-l}^be^{-\Phi(q)u}\mathbb{W}^{(q)}(u-b+l)\diff u \\
&=e^{-\Phi(q)(b-l)}+\delta\Phi(q) e^{- \Phi(q) (b-l)}\int_{0}^l e^{-\Phi(q)u}\mathbb{W}^{(q)}(u)\diff u.
\end{split}
\label{tilde_r_q}
\end{align}
In addition, define, for any $p\geq0$,  $q\geq-p$, and $a > b$, 
\begin{align} \label{cap_R_p_q}
\begin{split}
\mathcal{R}^{(p,q)}(a)&:=(1+\delta \mathbb{W}(0))W^{(q)}(a)+\delta\int_0^b\mathbb{W}^{(p+q)\prime}_+(u)W^{(q)}(a-u)\diff u \\
&+\frac{p}{q}\mathbb{W}^{(p+q)}(b)Z^{(q)}(a-b)+p\int_0^b\mathbb{W}^{(p+q)}(u)W^{(q)}(a-u)\diff u,
\end{split}
\end{align}
where in particular
\begin{align*}
\mathcal{R}^{(0,q)}(a)&=(1+\delta \mathbb{W}(0))W^{(q)}(a)+\delta\int_0^b\mathbb{W}_+^{(q)\prime}(u)W^{(q)}(a-u)\diff u \\
&= r^{(q)}(a;a) + \delta W^{(q)}(a-b) \mathbb{W}^{(q)} (b).
\end{align*}

For the case of bounded variation,
by \eqref{overshoot_reflected} and Lemma \ref{lemma_useful_identitysp} (ii) and (ii'), we get, for  $q \geq 0$, $p \geq -q$, and $0 \leq x \leq b < a$, 
\begin{align} \label{Z_overshoot_p_q}
\begin{split}
	\e_x&\bigg(e^{-(p+q)\eta_b^+}Z^{(q)}(a-U_{\eta_b^+})\bigg)\\
	&=\int_0^{b}\int_{(b-y,\infty)} Z^{(q)}(a-y-h)\bigg\{\frac{\mathbb{W}^{(p+q)}(b-x)}{\mathbb{W}_+^{(p+q)\prime}(b)}\mathbb{W}^{(p+q)\prime}_+(y)-\mathbb{W}^{(p+q)}(y-x)\bigg\}\Pi(\diff h)\diff y \\
&\qquad + \mathbb{W}(0)\int_{(b,\infty)} Z^{(q)}(a-h)\frac{\mathbb{W}^{(p+q)}(b-x)}{\mathbb{W}_+^{(p+q)\prime}(b)}\Pi(\diff h)\\
	&=\int_0^{b}\int_{(z,\infty)} Z^{(q)}(a-b+z-h)\bigg\{\frac{\mathbb{W}^{(p+q)}(b-x)}{\mathbb{W}_+^{(p+q)\prime}(b)}\mathbb{W}^{(p+q)\prime}_+(b-z)-\mathbb{W}^{(p+q)}(b-z-x)\bigg\}\Pi(\diff h)\diff z \\
&\qquad + \mathbb{W}(0)\int_{(b,\infty)} Z^{(q)}(a-h)\frac{\mathbb{W}^{(p+q)}(b-x)}{\mathbb{W}_+^{(p+q)\prime}(b)}\Pi(\diff h)\\
&=Z^{(q)}(a-x)+p\overline{\mathbb{W}}^{(p+q)}(b-x)+q\int_x^b\mathbb{W}^{(p+q)}(u-x)\left(p\overline{W}^{(q)}(a-u)+\delta W^{(q)}(a-u)\right)\diff u \\
&\qquad -q\frac{\mathbb{W}^{(p+q)}(b-x)}{\mathbb{W}_+^{(p+q)\prime}(b)}\mathcal{R}^{(p,q)}(a),
\end{split}
\end{align}
where the last equality holds by setting $\tilde{p} = p+q$ and $\tilde{q} = q$ in Lemma \ref{lemma_useful_identitysp} (ii') with $\alpha=0$, $\beta = a-b$ and $\gamma = a$ and, in Lemma \ref{lemma_useful_identitysp} (ii), with $\alpha=0$, $\beta = a-b$ and $\gamma = a-x$.
Its special case when $p=0$ reduces to
\begin{align*}
	\e_x\bigg(e^{-q\eta_b^+}Z^{(q)}(a-U_{\eta_b^+})\bigg)&=Z^{(q)}(a-x)+q \delta \int_x^b\mathbb{W}^{(q)}(u-x) W^{(q)}(a-u) \diff u-q\frac{\mathbb{W}^{(q)}(b-x)}{\mathbb{W}_+^{(q)\prime}(b)}\mathcal{R}^{(0,q)}(a).\notag
\end{align*} 

Taking a right-hand derivative of  \eqref{cap_R_p_q}, we have for $a > b$ 
\begin{align}
\mathcal{R}_+^{(p,q)\prime}(a)&=(1+\delta \mathbb{W}(0))W_+^{(q) \prime}(a)+\delta\int_0^b\mathbb{W}^{(p+q)\prime}_+(u)W^{(q)\prime}_+(a-u)\diff u\notag\\
&+p \Big( \mathbb{W}^{(p+q)}(b)W^{(q)}(a-b)+\int_0^b\mathbb{W}^{(p+q)}(y)W^{(q)\prime}_+(a-y)\diff y \Big), \notag
\end{align} 
where we used dominated convergence to interchange the derivative over the integral. Indeed, for all $0 < \epsilon < \bar{\epsilon}$ with fixed $\bar{\epsilon} > 0$,
\begin{align*}
\int_0^b\mathbb{W}^{(p+q)\prime}_+(u) \Big| \frac {W^{(q)}(a+\epsilon-u) - W^{(q)}(a -u) } {\epsilon} \Big|\diff u \leq \int_0^b\mathbb{W}^{(p+q)\prime}_+(u) \sup_{y \in [a-b, a+\bar{\epsilon}]}\Big|  W^{(q) \prime}_+(y)  \Big|\diff u \\
=  (\mathbb{W}^{(p+q)}(b)  - \mathbb{W}^{(p+q)}(0) ) \sup_{y \in [a-b, a+\bar{\epsilon}]}\Big| W^{(q)\prime}_+(y)  \Big|,
\end{align*} 
where supremum term is finite because $W^{(q)}_+(y)$ is finite if $y > 0$, as discussed in Remark \ref{remark_differentiability}.

On the other hand if $r^{(q) \prime}_{+}(l;a)$ is the right-hand derivative of \eqref{tilde_r_q} with respect to $l$, then for $a > b$ and $l \in \R$ 
  \begin{align*}r^{(q) \prime}_+(l;a)&=(1+\delta \mathbb{W}(0)) W^{(q)\prime}_+(l) + \delta\int_{a-b}^{l}\mathbb{W}^{(q)\prime}_+(l-z)W^{(q)\prime}_+(z) \diff z \\
 &= (1+\delta \mathbb{W}(0)) W^{(q)\prime}_+(l) + \delta\int_0^{l-a+b}\mathbb{W}^{(q)\prime}_+(u)W^{(q)\prime}_+(l-u) \diff u,
\end{align*}
where again the derivative can go into the integral by dominated convergence.
Hence, for $a > b$, 
\begin{align}r^{(q) \prime}_+(a;a)
 = \mathcal{R}_+^{(0,q)\prime}(a)= (1+\delta \mathbb{W}(0)) W^{(q)\prime}_+(a) + \delta\int_0^{b}\mathbb{W}^{(q)\prime}_+(u)W^{(q)\prime}_+(a-u) \diff u. \label{r_prime_a_a}
\end{align}
Using these and similarly to \eqref{Z_overshoot_p_q}, we can write, for $0 \leq x \leq b < a$, 
$q \geq 0$, and $p \geq -q$,
\begin{align} \label{W_overshoot_p_q}
\begin{split}
&\e_x\Big(e^{-(p+q)\eta_b^+}W^{(q)}(a-U_{\eta_b^+})\Big)\\ &=\int_0^{b}\int_{(b-y,\infty)} W^{(q)}(a-y-h)\bigg\{\frac{\mathbb{W}^{(p+q)}(b-x)}{\mathbb{W}_+^{(p+q)\prime}(b)}\mathbb{W}_+^{(p+q)\prime}(y)-\mathbb{W}^{(p+q)}(y-x)\bigg\}\Pi(\diff h)\diff y\\
&\qquad +\mathbb{W}(0)\int_{(b,\infty)} W^{(q)}(a-h)\frac{\mathbb{W}^{(p+q)}(b-x)}{\mathbb{W}_+^{(p+q)\prime}(b)} \Pi(\diff h)\\
&=W^{(q)}(a-x)+\int_x^b\mathbb{W}^{(p+q)}(u-x)\left(pW^{(q)}(a-u)+\delta W^{(q)\prime}_+(a-u)\right)\diff u -\frac{\mathbb{W}^{(p+q)}(b-x)}{\mathbb{W}_+^{(p+q)\prime}(b)} \mathcal{R}^{(p,q)\prime}_+(a),
\end{split}
\end{align}
and, in particular when $p=0$,
\begin{align} \label{W_overshoot_exp_special}
\begin{split}
\e_x \Big(e^{-q\eta_b^+}W^{(q)}(a-U_{\eta_b^+})\Big)&=r^{(q)}(a-x;a)-\frac{\mathbb{W}^{(q)}(b-x)}{\mathbb{W}_+^{(q)\prime}(b)} r^{(q)\prime}_+(a;a).
\end{split}
\end{align}

\subsection{Computation of resolvents} 
We now compute the resolvent measure given by 
$\e_x\big(\int_0^{T_a^+}e^{-qt}1_{\{ V_t \in B \}}\diff t\big)$
for any $0 < b < a$ and $0 \leq x \leq a$, and as its corollary $\e_x\big(\int_0^{\infty}e^{-qt}1_{\{ V_t \in B \}}\diff t\big)$ for $x \geq 0$. Due to the reflection at the lower barrier $0$, it is clear that, for any $x < 0$, 
$\e_x\big(\int_0^{T_a^+}e^{-qt}1_{\{ V_t \in B \}}\diff t\big)=\e \big(\int_0^{T_a^+}e^{-qt}1_{\{ V_t \in B \}}\diff t\big)$.
\begin{theorem}[Resolvent]\label{resolsp} For any $q\geq0$,  $0 < b < a$, and $0 \leq x \leq a$, and Borel set $B$ on $[0,a]$,
\begin{align*} 
\begin{split}
\e_x\bigg(\int_0^{T_a^+}e^{-qt}1_{\{ V_t \in B \}}\diff t\bigg)
&=-\Gamma_b^{(q)}(b-x;B) +  \Gamma^{(q)\prime}_b(b;B)   \frac {r^{(q)}(a-x;a)} {r^{(q)\prime}_+(a;a)}  \\ &+ \int_{b}^a 1_{\{ u \in B\}} \Big(  r^{(q) \prime}_+(u;u) \frac {r^{(q)}(a-x;a) } {r^{(q)\prime}_+(a;a)}   - r^{(q)}(u-x;u) \Big)\diff u.
\end{split}
\end{align*}
\end{theorem}
\begin{proof}
\par For convenience, let us denote the left-hand side  by $f^{(q)}(x,a;B)$.

(i) We suppose $Y$ is of bounded variation. For $x > b$, by using Remark \ref{remark_connection_Y_U}, the strong Markov property, \eqref{laplace_in_terms_of_z}, and (\ref{resolvent_density}),
\begin{align}\label{iden6sp}
f^{(q)}(x,a;B)&=\E_x\Big(\int_0^{\tau_a^+\wedge\tau_b^-}e^{-qt}1_{\{X_t\in B\}}\diff t\Big)+\E_x\left(e^{-q\tau_b^-} 1_{\{\tau_b^-<\tau_a^+\}}\right) f^{(q)}(b,a;B)\notag\\
&=\Big(\overline{\Gamma}^{(q)}_a(a-b;B)+f^{(q)}(b,a;B) \Big) \frac{W^{(q)}(a-x)}{W^{(q)}(a-b)} -\overline{\Gamma}^{(q)}_a(a-x;B).
\end{align}
For $0 \leq x \leq b$, again by Remark \ref{remark_connection_Y_U}, the strong Markov property, and \eqref{iden6sp},
\begin{align*}
f^{(q)}(x,a;B)
&= \e_x\bigg(\int_0^{\eta_b^+}e^{-qt}1_{\{U_t\in B\}}\diff t\bigg)+ \frac {\overline{\Gamma}^{(q)}_a(a-b;B)+f^{(q)}(b,a;B)} {W^{(q)}(a-b)}  \e_x\bigg(e^{-q\eta_b^+}W^{(q)}(a-U_{\eta_b^+})\bigg)\notag \\
&\qquad - \e_x\bigg(e^{-q\eta_b^+} \overline{\Gamma}^{(q)}_a(a-U_{\eta_b^+};B) \bigg).
\end{align*}
The first two expectations can be computed by \eqref{rsr} and  \eqref{W_overshoot_exp_special}.
For the third expectation, by  \eqref{overshoot_reflected},  \eqref{W_overshoot_exp_special}, and Fubini's theorem,
\begin{align*}
&\e_x\Big(e^{-q\eta_b^+} \overline{\Gamma}^{(q)}_a(a-U_{\eta_b^+};B) \Big)  \\
&= \int_0^{b}\int_{(b-y, \infty)} \int_{y+h}^a W^{(q)}(z-y-h) 1_{\{ z \in B\}} \diff z\left\{\frac{\mathbb{W}^{(q)}(b-x)}{\mathbb{W}_+^{(q)\prime}(b)}\mathbb{W}^{(q)\prime}_+(y)-\mathbb{W}^{(q)}(y-x)\right\}\Pi(\diff h)\diff y\notag\\
&\quad +\mathbb{W}(0) \int_{(b,\infty)}\int_{h}^a W^{(q)}(z-h) 1_{\{ z \in B\}} \diff z \frac{\mathbb{W}^{(q)}(b-x)}{\mathbb{W}_+^{(q)\prime}(b)} \Pi(\diff h) \\
&= \int_{b}^a 1_{\{ z \in B\}} \int_0^{b}\int_{(b-y, \infty)} W^{(q)}(z-y-h)  \left\{\frac{\mathbb{W}^{(q)}(b-x)}{\mathbb{W}_+^{(q)\prime}(b)}\mathbb{W}^{(q)\prime}_+(y)-\mathbb{W}^{(q)}(y-x)\right\}\Pi(\diff h)\diff y \diff z \notag\\
&\quad +\mathbb{W}(0) \int_{b}^a 1_{\{ z \in B\}} \int_{(b,\infty)} W^{(q)}(z-h)   \frac{\mathbb{W}^{(q)}(b-x)}{\mathbb{W}_+^{(q)\prime}(b)} \Pi(\diff h) \diff z \\
&= \int_{b}^a 1_{\{ z \in B\}}   \e_x\Big(e^{-q\eta_b^+} W^{(q)}(z-U_{\eta_b^+}) \Big) \diff z  \\
&= \int_{b}^a 1_{\{ z \in B\}} \Big( r^{(q)}(z-x;z)-\frac{\mathbb{W}^{(q)}(b-x)}{\mathbb{W}_+^{(q)\prime}(b)} r^{(q) \prime}_+(z;z)\Big) \diff z.
\end{align*}
Summing up these,
\begin{align} \label{f_x_a_B}
\begin{split}
f^{(q)}(x,a;B)&=\frac{\mathbb{W}^{(q)}(b-x)}{\mathbb{W}_+^{(q)\prime}(b)}\Gamma^{(q)\prime}_b(b;B)-\Gamma_b^{(q)}(b-x;B)\\ &-\int_{b}^a 1_{\{ z \in B\}} \Big( r^{(q)}(z-x;z)-\frac{\mathbb{W}^{(q)}(b-x)}{\mathbb{W}_+^{(q)\prime}(b)} r^{(q) \prime}_+(z;z) \Big) \diff z\\
&+ \frac {\overline{\Gamma}_a^{(q)}(a-b;B)+f^{(q)}(b,a;B)} {W^{(q)}(a-b)}  \Big( r^{(q)}(a-x;a)-\frac{\mathbb{W}^{(q)}(b-x)}{\mathbb{W}_+^{(q)\prime}(b)} r^{(q)\prime}_+(a;a) \Big).
\end{split}
\end{align}
Setting $x=b$,
\begin{align*}
f^{(q)}(b,a;B)&=\frac{\mathbb{W}(0)}{\mathbb{W}_+^{(q)\prime}(b)}\Gamma^{(q)\prime}_b(b;B) -\int_{b}^a 1_{\{ z \in B\}} \Big( W^{(q)}(z-b)-\frac{\mathbb{W}(0)}{\mathbb{W}_+^{(q)\prime}(b)} r^{(q) \prime}_+(z;z)\Big) \diff z\\
&+ \frac {\overline{\Gamma}^{(q)}_a(a-b;B)+f^{(q)}(b,a;B)} {W^{(q)}(a-b)}  \Big( W^{(q)}(a-b)-\frac{\mathbb{W}(0)}{\mathbb{W}_+^{(q)\prime}(b)} r^{(q)\prime}_+(a;a) \Big).
\end{align*}
Hence, solving for $f^{(q)}(b,a;B)$, we obtain
\begin{align}\label{potentialbsp2}
f^{(q)}(b,a;B)&=\frac{W^{(q)}(a-b)}{r^{(q)\prime}_+(a;a)}\left(\Gamma^{(q)\prime}_b(b;B)+\int_b^a 1_{\{ u \in B\}} r^{(q) \prime}_+(u;u) \diff u\right)-\overline{\Gamma}^{(q)}_a(a-b;B).
\end{align}
For $x > b$, substituting \eqref{potentialbsp2} in \eqref{iden6sp}, we have the claim.
For $0 \leq x \leq b$, substituting \eqref{potentialbsp2} in \eqref{f_x_a_B}, 
\begin{align*} 
&f^{(q)}(x,a;B) \\&=\frac{\mathbb{W}^{(q)}(b-x)}{\mathbb{W}_+^{(q)\prime}(b)}\Gamma^{(q)\prime}_b(b;B)-\Gamma_b^{(q)}(b-x;B)-\int_{b}^a 1_{\{ z \in B\}} \Big( r^{(q)}(z-x;z)-\frac{\mathbb{W}^{(q)}(b-x)}{\mathbb{W}_+^{(q)\prime}(b)} r^{(q) \prime}_+(z;z) \Big) \diff z\\
&\quad +  \Big(\Gamma^{(q)\prime}_b(b;B)+\int_b^a 1_{\{ u \in B\}} r^{(q) \prime}_+(u;u) \diff u\Big)   \Big( \frac {r^{(q)}(a-x;a)} {r^{(q)\prime}_+(a;a)} -\frac{\mathbb{W}^{(q)}(b-x)}{\mathbb{W}_+^{(q)\prime}(b)}  \Big) \\
&=-\Gamma^{(q)}_b(b-x;B) + \int_{b}^a 1_{\{ u \in B\}} \Big(  r^{(q) \prime}_+(u;u) \frac {r^{(q)}(a-x;a)} {r^{(q)\prime}_+(a;a)}   - r^{(q)}(u-x;u) \Big)\diff u +\Gamma^{(q)\prime}_b(b;B)    \frac {r^{(q)}(a-x;a)} {r^{(q)\prime}_+(a;a)},
\end{align*}
as desired.

(ii) 
 For the case $Y$ is of unbounded variation, we take a limit for the strongly approximating sequence $(Y^{(n)}; n \geq 1)$ of bounded variation.   Thanks to Proposition  \ref{prop_approximation} 
and Remark \ref{remark_strongly_approximating}, together with dominated convergence, the result holds essentially in the same way as in the spectrally negative case (Theorem 4.1 of \cite{YP20015}).  

To be more precise, for $q > 0$, using the same arguments as in Lemma 4.1 of \cite{YP20015}, $\p_x ( V_t^{(n)} = y )=\p_x ( V_t = y )=0$ for any a.e.\ $t > 0$ and $y \in (0,\infty)$ and $\p_x ( \sup_{0 \leq s \leq t} V_s = a ) = 0$ any a.e. $t > 0$. Hence, by  the same arguments as in Theorem 4.1 of \cite{YP20015}, we have the result when $0 \notin B$. Note that the hitting times $T_{a,n}^{+}$ defined in \eqref{def_T_hitting} for $V^{(n)}$ converge a.s.\ to $T_{a}^+$ similarly to the convergence of $\eta_{a,n}^+$ as shown in Remark \ref{remark_strongly_approximating}. 

For the resolvent at $\{0 \}$, because that of the reflected process $U$ does not have an atom for the unbounded variation case (see \eqref{rsr}), following the proof of Lemma 4.1 of \cite{YP20015} gives $\p_x ( V_t = 0 ) = 0$.  In view of the expression for the bounded variation case, by  Remark \ref{remark_strongly_approximating} and in particular by the fact that $\mathbb{W}_n^{(q)}(0) \xrightarrow{n \uparrow \infty} \mathbb{W}^{(q)}(0) = 0$, we can confirm that $\e_x\big(\int_0^{T_{a,n}^{+}}e^{-qt}1_{\{ V_t^{(n)} =0 \}}\diff t\big) \xrightarrow{n \uparrow \infty} 0 = \e_x\big(\int_0^{T_a^+}e^{-qt}1_{\{ V_t =0 \}}\diff t\big)$. This completes the proof for $q > 0$. For the case $q = 0$, the claim holds by monotone convergence and the continuity of the scale functions in $q$.

\end{proof}
\begin{remark}
In particular, when $\delta = 0$ in Theorem \ref{resolsp}, we recover \eqref{rsr}.
\end{remark}

We shall now take $a$ to $\infty$ in Theorem \ref{resolsp}.
By differentiating \eqref{tilde_r_q}, we obtain
\begin{align} \label{r_tilde_derivative}
\begin{split}
\hat{r}^{(q) \prime}(b-x)&=   \Phi(q) \Big( e^{-\Phi(q)x} (1+\delta \mathbb{W}(0))+\delta \int_x^be^{-\Phi(q)u}\mathbb{W}_+^{(q)\prime}(u-x)\diff u \Big)\\
&=   \Phi(q) \Big( e^{-\Phi(q)x}+\delta \Big(e^{-\Phi(q) b} \mathbb{W}^{(q)} (b-x) + \Phi(q) \int_x^be^{-\Phi(q)u}\mathbb{W}^{(q)}(u-x)\diff u \Big) \Big).
\end{split}
\end{align}
In particular,
\begin{align} \label{r_tilde_derivative_at_b}
\begin{split}
\hat{r}^{(q) \prime}(b)&=   \Phi(q)  \Big( (1+\delta \mathbb{W}(0))+\delta \int_0^be^{-\Phi(q)u}\mathbb{W}^{(q)\prime}_+(u)\diff u \Big) \\
&=  \Phi(q) \Big (1 +\delta \Big(e^{-\Phi(q) b} \mathbb{W}^{(q)} (b) + \Phi(q) \int_0^be^{-\Phi(q)u}\mathbb{W}^{(q)}(u)\diff u \Big) \Big).
\end{split}
\end{align}

By \eqref{W_q_limit}, if $q > 0$, or $q = 0$ and $\psi'_{X,+}(0) < 0$ (or $X_t \xrightarrow{t \uparrow \infty} \infty$), then $\Phi(q) > 0$ and hence
\begin{align}
\frac{r^{(q)}(a-x;a)}{r^{(q)\prime}_+(a;a)} \xrightarrow{a \uparrow \infty} \frac{\hat{r}^{(q)}(b-x)}{\hat{r}^{(q)\prime}(b)}. \label{convergence_r_ratio}
\end{align}

Applying this in Theorem \ref{resolsp} we get the following corollary. 


\begin{corollary}\label{resolinftysp} Fix $x \geq 0$ and any Borel set $B$ on $[0, \infty)$.

(i) If $q > 0$, or $q = 0$ and $\psi'_{X,+}(0)  < 0$ (or $X_t \xrightarrow{t \uparrow \infty} \infty$), then
\begin{align*}
\e_x\bigg(\int_0^{\infty}e^{-qt}1_{\{ V_t \in B \}}\diff t\bigg)
&=\frac{\hat{r}^{(q)}(b-x)}{\hat{r}^{(q)\prime}(b)}\Gamma_b^{(q)\prime}(b;B)- \Gamma_b^{(q)}(b-x; B) \notag \\
	&+\int_b^{\infty} 1_{\{u \in B \}}\left(\frac{\hat{r}^{(q)}(b-x)}{\hat{r}^{(q)\prime}(b)}r^{(q)\prime}_+(u;u)-r^{(q)}(u-x;u)\right)\diff u.
\end{align*}
(ii) If $q=0$ and $\psi'_{X,+}(0)  \geq 0$, it is infinity given $Leb(B) > 0$.

\end{corollary}
\begin{proof}  (i) For the case $B$ is contained in a compact set, say $[0,K]$, dominated convergence and the convergence \eqref{convergence_r_ratio} show the result.
%
%
%
%
For the case $B$ is unbounded,  taking $K \uparrow \infty$ via  monotone convergence  shows the result.

(ii) In this case $\Phi(0) = 0$, and hence, in view of \eqref{r_tilde_derivative_at_b}, $\hat{r}^{(q)\prime}(b) \xrightarrow{q \downarrow 0} 0$. Hence, taking $q \downarrow 0$ in the result from (i) shows the result by monotone convergence.

\end{proof}

As another corollary, 
we provide the one-sided exit problem.
\begin{corollary}[One-sided exit problem]\label{rrosep}
For any $q\geq0$, $0 < b < a$, and $0 \leq x \leq a$, we have 
\begin{align*}
\mathbb{E}_x\left(e^{-qT_a^+} \right)
&=Z^{(q)}(a-x)+q\delta\int_x^b\mathbb{W}^{(q)}(y-x)W^{(q)}(a-y)\diff y -q\frac{r^{(q)}(a-x;a)}{r^{(q)\prime}_+(a;a)} \mathcal{R}^{(0,q)}(a).
\end{align*} 
In particular, $T_a^+ < \infty$ $\p_x$-a.s.\ for any $0 < b < a$ and $0 \leq x \leq a$. \end{corollary}
\begin{proof}

By Theorem \ref{resolsp} and \eqref{r_prime_a_a}, we have
\begin{multline} \label{one_sided_0_T_a}
\e_x\bigg(\int_0^{T_a^+}e^{-qt}\diff t\bigg)=-\Gamma_b^{(q)}(b-x;[0,a]) \\+ \big(\mathcal{R}^{(0,q)}(a)-\mathcal{R}^{(0,q)}(b) +  \Gamma^{(q)\prime}_b(b;[0,a]) \big)   \frac {r^{(q)}(a-x;a)} {r^{(q)\prime}_+(a;a)}   - \int_{b}^a  r^{(q)}(u-x;u) \diff u.
\end{multline}
Here $\Gamma_b^{(q)}(b-x;[0,a]) = \overline{\mathbb{W}}^{(q)} (b-x)$ and $\Gamma^{(q)\prime}_b(b;[0,a]) = \mathbb{W}^{(q)} (b)$.  In addition, by differentiating \eqref{RLqp}, 
\begin{align*}
\mathcal{R}^{(0,q)}(b) 
= (1+\delta \mathbb{W}(0))W^{(q)}(b)+\mathbb{W}^{(q)} (b) - W^{(q)} (b) - \delta \mathbb{W} (0) W^{(q)} (b) =\mathbb{W}^{(q)} (b).
\end{align*}
Finally, by Fubini's theorem and \eqref{RLqp},
\begin{align*}
 \int_{b}^a  r^{(q)}(u-x;u) \diff u &= \int_b^a \Big( W^{(q)}(u-x)+\delta\int_{x}^{b}\mathbb{W}^{(q)}(z-x)W^{(q)\prime}_+(u-z) \diff z \Big) \diff u \\
 &= \overline{W}^{(q)} (a-x) - \overline{W}^{(q)} (b-x)  +    \delta\int_{x}^{b}\mathbb{W}^{(q)}(z-x) \big(W^{(q)} (a-z) - W^{(q)}(b-z) \big) \diff z  \\
 &= \overline{W}^{(q)} (a-x) +    \delta\int_{x}^{b}\mathbb{W}^{(q)}(z-x) W^{(q)} (a-z)  \diff z  - \overline{\mathbb{W}}^{(q)}(b-x). 
\end{align*}
Substituting these in \eqref{one_sided_0_T_a}, we have
\begin{align*} 
\e_x\bigg(\int_0^{T_a^+}e^{-qt}\diff t\bigg)
=  \frac {r^{(q)}(a-x;a)} {r^{(q)\prime}_+(a;a)}   \mathcal{R}^{(0,q)}(a) 
- \overline{W}^{(q)} (a-x) -   \delta\int_{x}^{b}\mathbb{W}^{(q)}(z-x) W^{(q)} (a-z)  \diff z.  \end{align*}
Now the result is immediate by the identity $\E_x (e^{-q T_a^+}) = 1 - q\e_x\big(\int_0^{T_a^+}e^{-qt}\diff t\big)$.
The second claim holds by taking $q \downarrow 0$.
\end{proof}

\begin{remark}  In view of Corollary \ref{rrosep}, when $\delta = 0$, we recover \eqref{upcrossing_time_reflected}.
\end{remark}

\subsection{Dividends}\label{dividendsp}
Recall the cumulative amount of dividends $L$ as defined in \eqref{L_def}. 
We compute their expected NPV using Theorem \ref{resolsp} and Corollary \ref{resolinftysp}.

%
%
%
%
%
\begin{corollary}
For any $q\geq0$, $0 < b < a$, and $0 \leq x \leq a$, we have
\begin{align*}
\mathbb{E}_x\Big(\int_0^{T_a^+}e^{-qt} \diff L_t\Big)&=  \delta  \Big( \overline{\mathbb{W}}^{(q)}(b-x)-\overline{W}^{(q)}(a-x)-\delta\int_x^b\mathbb{W}^{(q)}(y-x)W^{(q)}(a-y)\diff y \Big) \\
&+ \delta \frac{r^{(q)}(a-x;a)}{r^{(q)\prime}_+(a;a)}\left(\mathcal{R}^{(0,q)}(a) -\mathbb{W}^{(q)}(b)\right).
\end{align*}
\end{corollary}
\begin{proof}
We shall show for the case $q > 0$; the case $q = 0$ then holds by monotone convergence and the continuity of the scale functions in $q$. We have
\begin{align*}
\delta^{-1}\mathbb{E}_x\left(\int_0^{T_a^+}e^{-qt} \diff L_t\right)= \E_x\left(\int_0^{T_a^+}e^{-qt}1_{\{V_t>b\}}\diff t\right) = \frac {1 - \E_x(e^{-qT_a^+})} q - \E_x\left(\int_0^{T_a^+}e^{-qt}1_{\{V_t \leq b\}}\diff t\right). 
\end{align*}
Here, using Theorem \ref{resolsp}, we obtain
\begin{align*}
\mathbb{E}_x\left(\int_0^{T_a^+}e^{-qt}1_{\{V_t \leq b\}}\diff t\right)&= \mathbb{W}(0)\frac{r^{(q)}(a-x;a)}{r^{(q)\prime}_+(a;a)}+\int_0^{b}\Big(\frac{r^{(q)}(a-x;a)}{r^{(q)\prime}_+(a;a)}\mathbb{W}^{(q)\prime}_+(u)-\mathbb{W}^{(q)}(u-x)\Big)\diff u \\
&= \mathbb{W}^{(q)}(b)\frac{r^{(q)}(a-x;a)}{r^{(q)\prime}_+(a;a)}-\overline{\mathbb{W}}^{(q)}(b-x).
\end{align*}
By this and Corollary \ref{rrosep},
\begin{align*}
&\E_x\bigg(\int_0^{T_a^+}e^{-qt}1_{\{V_t>b\}}\diff t\bigg) \\ &= - \frac 1 q \Big( Z^{(q)}(a-x)+q\delta\int_x^b\mathbb{W}^{(q)}(y-x)W^{(q)}(a-y)\diff y-q\frac{r^{(q)}(a-x;a)}{r^{(q)\prime}_+(a;a)}\mathcal{R}^{(0,q)}(a)\Big) \\ &- \mathbb{W}^{(q)}(b)\frac{r^{(q)}(a-x;a)}{r^{(q)\prime}_+(a;a)}+\frac {\mathbb{Z}^{(q)}(b-x)} q \\
&=\Big( \overline{\mathbb{W}}^{(q)}(b-x)-\overline{W}^{(q)}(a-x)-\delta\int_x^b\mathbb{W}^{(q)}(y-x)W^{(q)}(a-y)\diff y \Big) +\frac{r^{(q)}(a-x;a)}{r^{(q)\prime}_+(a;a)}\left(\mathcal{R}^{(0,q)}(a) -\mathbb{W}^{(q)}(b)\right).
\end{align*}
Multiplying by $\delta$, we have the claim.
\end{proof}
\begin{corollary} \label{corollary_L}
For any $q>0$ and $x \geq 0$,
\begin{align}
&\mathbb{E}_x\left(\int_0^{\infty}e^{-qt} \diff L_t\right) = \delta \Big( \frac {\mathbb{Z}^{(q)}(b-x)} q -  \frac{\hat{r}^{(q)}(b-x)}{\hat{r}^{(q)\prime}(b)} \mathbb{W}^{(q)} (b) \Big). \label{L_NPV_infty}
\end{align}
For $q = 0$, it is infinity.
\end{corollary}
\begin{proof}
(i) Suppose $q > 0$. We have 
\begin{align} \label{L_NPV_decomp}
\delta^{-1}\mathbb{E}_x\left(\int_0^{\infty}e^{-qt} \diff L_t\right)=\mathbb{E}_x\left(\int_0^{\infty}e^{-qt}1_{\{V_t >b\}}\diff t \right) = \frac 1 q -  \mathbb{E}_x\left(\int_0^{\infty}e^{-qt}1_{\{V_t<b\}}\diff t\right).
\end{align}
Here, using Corollary \ref{resolinftysp},
\begin{align} \label{V_below_b_temp}
\begin{split}
\mathbb{E}_x\left(\int_0^{\infty}e^{-qt}1_{\{V_t <b\}}\diff t\right)&= \mathbb{W}(0)\frac{\hat{r}^{(q)}(b-x)}{\hat{r}^{(q)\prime}(b)}+\frac{\hat{r}^{(q)}(b-x)}{\hat{r}^{(q)\prime}(b)}  \int_0^b\mathbb{W}_+^{(q)\prime}(u)\diff u- \int_0^b\mathbb{W}^{(q)}(u-x) \diff u \\
&=\frac{\hat{r}^{(q)}(b-x)}{\hat{r}^{(q)\prime}(b)}\mathbb{W}^{(q)}(b)-\overline{\mathbb{W}}^{(q)}(b-x).
\end{split}
\end{align}
Substituting this in \eqref{L_NPV_decomp},  the result holds for the case $q > 0$.

(ii) Suppose $q = 0$.  For the case $\psi'_{X,+}(0)  \geq 0$, it is infinity by Corollary \ref{resolinftysp}.
Otherwise, by the convergence \eqref{convergence_r_ratio}, 
the right-hand side of \eqref{L_NPV_infty} goes to infinity as $q \rightarrow 0$.  This together with
 monotone convergence applied to the left-hand side of \eqref{L_NPV_infty} completes the proof.  
\end{proof}
%

\subsection{Costs of capital injection}
Recall that $R$ models the cumulative amount of capital injection.  We shall obtain its expected NPV as follows.  Note that the result can be extended to  the case $x < 0$ because due to the reflection at  $0$, we have $\mathbb{E}_x (\int_{[0,T_a^+]}e^{-qt}\diff R_t ) = \mathbb{E} (\int_{[0,T_a^+]}e^{-qt}\diff R_t ) + |x|$.
\begin{proposition}\label{cispp}
For any $q \geq 0$, $0 < b < a$, and $0 \leq x \leq a$,
\begin{align}
\mathbb{E}_x\left(\int_{[0,T_a^+]}e^{-qt}\diff R_t\right)=\frac {r^{(q)}(a-x;a)}{r^{(q)\prime}_+(a;a)}.\notag
\end{align}
\end{proposition}
\begin{proof}
Let us denote the left-hand side by $g^{(q)}(x,a)$.

(i) Suppose $Y$ is of bounded variation. For $x > b$, by an application of Remark \ref{remark_connection_Y_U}, the strong Markov property, \eqref{laplace_in_terms_of_z}, and because $R$ does not increase on $[0,T_a^+\wedge T_b^-]$, we get 
\begin{align}\label{cisp0}
g^{(q)}(x,a)
&=g^{(q)}(b,a)\frac{W^{(q)}(a-x)}{W^{(q)}(a-b)}.
\end{align}
Now let us consider the case $0\leq x\leq b$: again by Remark \ref{remark_connection_Y_U}, the strong Markov property, (\ref{capital_injection_identity_SP}), \eqref{W_overshoot_exp_special}, and \eqref{cisp0},
\begin{align} \label{g_x_a}
\begin{split}
g^{(q)}(x,a) 
&= \mathbb{E}_x\Big(\int_{[0,\eta_b^+]}e^{-qt}\ud \widetilde{R}_t\Big)+\E_x\left(e^{-q \eta_b^+}g^{(q)}(U_{\eta_b^+},a)\right) \\
&=\frac{\mathbb{W}^{(q)}(b-x)}{\mathbb{W}_+^{(q)\prime}(b)} +\frac{g^{(q)}(b,a)}{W^{(q)}(a-b)} \Big( r^{(q)}(a-x;a) -\frac{\mathbb{W}^{(q)}(b-x)}{\mathbb{W}_+^{(q)\prime}(b)} r^{(q)\prime}_+(a;a) \Big).
\end{split}
\end{align}
By setting $x = b$ and solving  
 for $g^{(q)}(b,a)$, we have
\begin{equation}\label{cisp2}
g^{(q)}(b,a)=\frac{W^{(q)}(a-b)}{r^{(q)\prime}_+(a;a)}.
\end{equation}
Substituting \eqref{cisp2} in \eqref{cisp0}, we have the result for $x > b$.
On the other hand, for $0 \leq x < b$, substituting (\ref{cisp2}) in \eqref{g_x_a} gives
\begin{align}
g^{(q)}(x,a)
&=\frac{\mathbb{W}^{(q)}(b-x)}{\mathbb{W}_+^{(q)\prime}(b)}+\frac{1}{r^{(q)\prime}_+(a;a)}\bigg(r^{(q)}(a-x;a)-\frac{\mathbb{W}^{(q)}(b-x)}{\mathbb{W}_+^{(q)\prime}(b)}r^{(q)\prime}_+(a;a)\bigg)=\frac{r^{(q)}(a-x;a)}{r^{(q)\prime}_+(a;a)}.\notag
\end{align}
This completes the result for the bounded variation case.

(ii) For the case of unbounded variation, using Proposition \ref{prop_approximation} and Remark \ref{remark_strongly_approximating} we can follow the same steps as in the proof for the spectrally negative case (see Proposition 5.1 of \cite{YP20015}).   The proof for the spectrally positive case is easier because there is no need of taking care of the overshoot at the down-crossing time, and therefore the proof does not require that the first moment of $Y$ is finite.
\end{proof}
\begin{remark}
By taking $\delta=0$ in Proposition \ref{cispp}, we recover \eqref{capital_injection_identity_SP}.
\end{remark}
By taking $a$ to $\infty$ in Proposition \ref{cispp} we get the following. The proof is similar to that of Corollary \ref{resolinftysp} and is hence omitted. 
\begin{corollary} \label{cispp_infty}
Fix $x \geq 0$. If $q > 0$ or $q=0$ and $\psi'_{X,+}(0)  < 0$ (or $X_t \xrightarrow{t \uparrow \infty} \infty$ a.s.),
\begin{align}
\mathbb{E}_x\left(\int_{[0, \infty)}e^{-qt}\diff R_t\right)= \frac{\hat{r}^{(q)}(b-x)}{\hat{r}^{(q)\prime}(b)}.
\end{align}
Otherwise it is infinity.
\end{corollary}
%

\subsection{Occupation times for dividend payouts}  
We conclude this section by computing  the occupation time of the process $V$ above and below the level of refraction $b > 0$. Namely, we compute the Laplace transforms of the following quantities: 
\[
\int_0^{T_a^+}1_{\{V_s<b\}}\diff s\qquad\text{and}\qquad \int_0^{T_a^+}1_{\{V_s > b\}}\diff s.
\]
Similarly to the case of resolvent and capital injection, these identities can be extended to the case $x < 0$.
\begin{proposition}\label{occupation_time_below}
For any $p\geq0$, $q\geq-p$, $0 < b < a$, and $0 \leq x \leq a$,
\begin{align} \label{occupation_identity1}
\begin{split}
\E_x&\Big(e^{-qT_a^+-p\int_0^{T_a^+}1_{\{V_s<b\}}\diff s}\Big)\\
&=Z^{(q)}(a-x)+p\overline{\mathbb{W}}^{(p+q)}(b-x)+q\int_x^b\mathbb{W}^{(p+q)}(u-x)\left(p\overline{W}^{(q)}(a-u)+\delta W^{(q)}(a-u)\right)\diff u\\
&-q\frac{\mathcal{R}^{(p,q)}(a)}{\mathcal{R}^{(p,q)\prime}_+(a)}\bigg(W^{(q)}(a-x)+\int_x^b\mathbb{W}^{(p+q)}(u-x)\left(pW^{(q)}(a-u)+\delta W_+^{(q)\prime}(a-u)\right)\diff u\bigg),
\end{split}
\end{align}
and
\begin{align} \label{occupation_identity2}
\begin{split}
	\E_x&\Big(e^{-qT_a^+-p\int_0^{T_a^+}1_{\{V_s>b\}}\diff s}\Big) \\
	&=Z^{(p+q)}(a-x)-p\overline{\mathbb{W}}^{(q)}(b-x)+(p+q)\int_x^b\mathbb{W}^{(q)}(u-x)\left(-p\overline{W}^{(p+q)}(a-u)+\delta W^{(p+q)}(a-u)\right)\diff u \\
	&-(p+q)\frac{\mathcal{R}^{(-p,p+q)}(a)}{\mathcal{R}^{(-p,p+q)\prime}_+(a)}\bigg(W^{(p+q)}(a-x)+\int_x^b\mathbb{W}^{(q)}(u-x)\left(-pW^{(p+q)}(a-u)+\delta W^{(p+q)\prime}_+(a-u)\right)\diff u\bigg).
	\end{split}
\end{align}

\end{proposition}
\begin{proof}

We shall prove  \eqref{occupation_identity1}. The equality \eqref{occupation_identity2} holds by observing that
\[
\E_x\Big(e^{-qT_a^+-p\int_0^{T_a^+}1_{\{V_s>b\}}\diff s}\Big)=
\E_x\Big(e^{-(p+q)T_a^++p\int_0^{T_a^+}1_{\{V_s<b\}}\diff s}\Big).
\]

Let us denote the left-hand side of \eqref{occupation_identity1}  by $h^{(p,q)}(x,a)$.
Here we focus on the case of bounded variation; it can be extended to the unbounded variation case by dominated convergence, Proposition \ref{prop_approximation}, and Remark \ref{remark_continuity_theorem_scale_function}.

For the case $x > b$, by Remark \ref{remark_connection_Y_U}, the strong Markov property, and \eqref{laplace_in_terms_of_z}, 
\begin{align}\label{ot1sp0}
h^{(p,q)}(x,a)&=\E_x \big(e^{-q\tau_a^+} 1_{\{ \tau_a^+<\tau_b^- \}} \big)+\E_x \big(e^{-q\tau_b^-}h^{(p,q)}(b,a) 1_{\{\tau_b^-<\tau_a^+\}} \big)\notag\\
&=Z^{(q)}(a-x) + \big(h^{(p,q)}(b,a) -Z^{(q)}(a-b) \big) \frac{W^{(q)}(a-x)}{W^{(q)}(a-b)}.
\end{align}

Now, for the case $0\leq x \leq b$, we obtain, again by Remark \ref{remark_connection_Y_U}, the strong Markov property,  \eqref{upcrossing_time_reflected}, \eqref{Z_overshoot_p_q}, and \eqref{W_overshoot_p_q},
\begin{align}\label{ot1sp1}
\begin{split}
&h^{(p,q)}(x,a)=\E_x \Big(e^{-(p+q)\eta_b^+}h^{(p,q)}(U_{\eta_b^+},a) \Big)\\
&= \E_x \big(e^{-(p+q)\eta_b^+}Z^{(q)}(a-U_{\eta_b^+}) \big) + \frac {h^{(p,q)}(b,a) -Z^{(q)}(a-b)} {W^{(q)}(a-b)}  \E_x(e^{-(p+q)\eta_b^+}W^{(q)}(a-U_{\eta_b^+})) \\
&= Z^{(q)}(a-x)+p\overline{\mathbb{W}}^{(p+q)}(b-x)+q\int_x^b\mathbb{W}^{(p+q)}(u-x)\left(p\overline{W}^{(q)}(a-u)+\delta W^{(q)}(a-u)\right)\diff u\\
&-q\frac{\mathbb{W}^{(p+q)}(b-x)}{\mathbb{W}_+^{(p+q)\prime}(b)}\mathcal{R}^{(p,q)}(a) + \frac {h^{(p,q)}(b,a) -Z^{(q)}(a-b)} {W^{(q)}(a-b)}  \\ &\times \Bigg( W^{(q)}(a-x)+\int_x^b\mathbb{W}^{(p+q)}(u-x)\left(pW^{(q)}(a-u)+\delta W^{(q)\prime}_+(a-u)\right)\diff u -\frac{\mathbb{W}^{(p+q)}(b-x)}{\mathbb{W}_+^{(p+q)\prime}(b)}\mathcal{R}_+^{(p,q)\prime}(a) \Bigg).
\end{split}
\end{align}
Setting $x = b$ and 
solving for $h^{(p,q)}(b,a)$, 
\begin{align} \label{h_at_b}
h^{(p,q)}(b,a) = Z^{(q)}(a-b)- q W^{(q)}(a-b)\frac {\mathcal{R}^{(p,q)}(a)} {\mathcal{R}^{(p,q)\prime}_+(a)}.			
\end{align}
Substituting \eqref{h_at_b} in \eqref{ot1sp0}, we have the claim for $x > b$.  On the other hand, substituting \eqref{h_at_b} in \eqref{ot1sp1}, we have the claim for $0 \leq x < b$.



\end{proof}

\section{On the dividend problem with capital injection} \label{section_dividend}

In this section, we use the results obtained in Section \ref{section_refracted_reflected} to solve the optimal dividend problem in the dual model driven by $Y$.  We consider a version  where the time horizon is infinity, and the shareholders are required to inject capital to prevent the company from going bankrupt, with an extra condition on the dividend strategy described below.  
 
A strategy  is a pair $\pi := \left( L_t^{\pi}, R_t^{\pi}; t \geq 0 \right)$ of nondecreasing, right-continuous, and adapted processes (with respect to the filtration generated by $Y$) such that $L_{0-}^\pi = R_{0-}^\pi = 0$ where $L^{\pi}$ is the cumulative amount of dividends and $R^{\pi}$ is that of injected capital. With $V_{0-}^\pi := x$ and $V_t^\pi := Y_t - L_t^\pi + R_t^\pi$, $t \geq 0$, it is required that $V_t^\pi \geq 0$ a.s.\ uniformly in $t$.  In addition,  with $\delta > 0$ fixed, we require that $L^\pi$ is absolutely continuous with respect to the Lebesgue measure of the form $L_t^\pi = \int_0^t \ell^\pi_s \diff s$, $t \geq 0$, with $\ell^\pi$ restricted to take values in $[0,\delta]$ uniformly in time.  

Assuming that $\beta > 1$ is the cost per unit injected capital and $q > 0$ is the discount factor, we want to maximize
\begin{align*}
v_{\pi} (x) := \mathbb{E}_x \left( \int_0^\infty e^{-q t} \ell_t^\pi  \diff t - \beta \int_{[0, \infty)} e^{-q t} \diff R_t^{\pi}\right), \quad x \geq 0.
\end{align*}
Hence the problem is to compute
\begin{equation*}
v(x):=\sup_{\pi \in \mathcal{A}}v_{\pi}(x), \quad x \geq 0,
\end{equation*}
where $\mathcal{A}$ is the set of all admissible strategies  that satisfy the constraints described above.

Our objective is to show the optimality of the refraction-reflection strategy, say $\pi^{b}$, with a suitable refraction level $b \geq 0$. Namely, dividends are paid whenever the surplus process is above $b$ at the maximal rate $\delta$ while it is pushed by capital injection whenever it attempts to down-cross zero.  It is clear that this strategy is admissible and its expected NPV is given by 
\begin{align} \label{v_pi}
v^b (x) := \mathbb{E}_x \left( \int_0^\infty e^{-q t} \diff L_t^{0,b} - \beta \int_{[0, \infty)} e^{-q t} \diff R_t^{0,b}\right), \quad x \geq 0,
\end{align}
where $L^{0,b}$ and $R^{0,b}$ are the processes studied in the previous section. Here, we add the superscripts to stress the dependence on the refraction level $b$, which we aim to choose.

\subsection{Candidate value function}
%

Fix $b > 0$. By Corollaries \ref{corollary_L} and \ref{cispp_infty},
\begin{align}
&v^b(x) = \delta  \frac {\mathbb{Z}^{(q)}(b-x)} q -   \frac{\hat{r}^{(q)}(b-x)}{\hat{r}^{(q) \prime}(b)} (\delta \mathbb{W}^{(q)} (b) +\beta ), \quad x \geq 0. \label{v_b_x}
\end{align}

Given the spectrally positive L\'evy process $Y$, we call a function $g$ (defined on the real line) \emph{sufficiently smooth} at $x > 0$ if $g$ is continuously differentiable when $Y$ has paths of bounded variation and is twice continuously differentiable  when it has paths of unbounded variation.

Here, we shall obtain a condition such that $v^b$ is sufficiently smooth at $b$.  We note by the continuity of the scale function that, regardless of the choice of $b$, $v^b(x)$ is sufficiently smooth on $(0, \infty) \backslash \{b\}$.

Suppose $x > b$.  Because $v^b(x) =  {\delta} / q -   {e^{-\Phi(q)x}} (\delta \mathbb{W}^{(q)} (b) +\beta ) /\hat{r}^{(q) \prime}(b)$,
\begin{align*}
v^{b \prime}(x) =  \Phi(q) \frac{e^{-\Phi(q)x}}{\hat{r}^{(q) \prime}(b)} (\delta \mathbb{W}^{(q)} (b) +\beta ) \quad \textrm{and} \quad
v^{b \prime \prime}(x) = - \Phi(q)^2 \frac{e^{-\Phi(q)x}}{\hat{r}^{(q) \prime}(b)} (\delta \mathbb{W}^{(q)} (b) +\beta ).
\end{align*}
By taking $x \downarrow b$, we have
\begin{align} \label{v_b_derivative_right}
v^{b \prime}_+(b) =  \Phi(q) \frac{e^{-\Phi(q)b}}{\hat{r}^{(q) \prime}(b)} (\delta \mathbb{W}^{(q)} (b) +\beta ) \quad \textrm{and} \quad
v^{b \prime\prime}_+(b) &=  -\Phi(q)^2 \frac{e^{-\Phi(q)b}}{\hat{r}^{(q) \prime}(b)} (\delta \mathbb{W}^{(q)} (b) +\beta ).
\end{align}

Suppose $x < b$, then by differentiating \eqref{v_b_x},
\begin{align}\label{derivative_vf}
v^{b\prime}(x) 
= -\delta  {\mathbb{W}^{(q)}(b-x)}  +  \frac{ \hat{r}^{(q) \prime} (b-x)}{\hat{r}^{(q) \prime}(b)} (\delta \mathbb{W}^{(q)} (b) +\beta ).
\end{align}
Therefore, by taking $x \uparrow b$,
\begin{align} \label{v_b_derivative_left}
v^{b\prime}_-(b) 
&= -\delta  {\mathbb{W}(0)}  +  \frac{ \Phi(q) e^{-\Phi(q)b} (1+\delta \mathbb{W}(0))}{\hat{r}^{(q) \prime}(b)} (\delta \mathbb{W}^{(q)} (b) +\beta ).
\end{align}
By matching \eqref{v_b_derivative_right} and \eqref{v_b_derivative_left},  we see that $v^b$ is differentiable at $b$ if and only if
\begin{align*}
0
&=   \mathbb{W} (0) \Big( \frac{ \Phi(q) e^{-\Phi(q)b} }{\hat{r}^{(q) \prime}(b)} (\delta \mathbb{W}^{(q)} (b) +\beta ) - 1 \Big).
\end{align*}
For the case of bounded variation (see \eqref{eq:Wqp0}), this is equivalent to 
\begin{align}\label{the_condition}
\frac{ \Phi(q) e^{-\Phi(q)b} }{\hat{r}^{(q) \prime}(b)} (\delta \mathbb{W}^{(q)} (b) +\beta ) =1. 
\end{align}
For the case of unbounded variation, because the differentiability automatically holds by $\mathbb{W}(0) = 0$, we shall further pursue twice continuous differentiability at $b$; for $x < b$, by differentiating  (\ref{derivative_vf}) and recalling \eqref{eq:Wqp0} and Remark  \ref{remark_smoothness},
\begin{align*}
v^{b\prime \prime}(x) 
&= \delta  {\mathbb{W}^{(q) \prime}(b-x)}  - \Phi(q) (\delta \mathbb{W}^{(q)} (b) +\beta ) \\ &\qquad \times\frac{ \Phi(q) e^{-\Phi(q)x}+\delta \Big(e^{-\Phi(q) b} \mathbb{W}^{(q) \prime} (b-x) + \Phi(q) \int_x^be^{-\Phi(q)u}\mathbb{W}^{(q) \prime}(u-x)\diff u \Big)}{\hat{r}^{(q) \prime}(b)}.
\end{align*}
In particular, if \eqref{the_condition} holds, then for $x>b$
	\begin{align*}
	v^{b\prime \prime}(x)
	&=   - \Phi(q) (\delta \mathbb{W}^{(q)} (b) +\beta )\frac{ \Phi(q) e^{-\Phi(q)x}+\delta  \Phi(q) \int_x^be^{-\Phi(q)u}\mathbb{W}^{(q) \prime}(u-x)\diff u }{\hat{r}^{(q) \prime}(b)},
\end{align*}
and therefore
\begin{align*}
	v^{b\prime \prime}_+(b)=-\Phi(q)^2 \frac{e^{-\Phi(q)b}}{\hat{r}^{(q) \prime}(b)} (\delta \mathbb{W}^{(q)} (b) +\beta )=v^{b\prime \prime}_-(b).
\end{align*}
Hence twice continuous differentiability at $b$ is satisfied if \eqref{the_condition} holds.

In order to see the existence of $b$ such that \eqref{the_condition} holds, with
\begin{align}
 f(b)   :=1  +\delta \Phi(q) \int_0^b e^{-\Phi(q) u} \mathbb{W}^{(q)}(u) \diff u - e^{-\Phi(q)b}  \beta, \quad b \geq 0, \label{f_b}
\end{align}
we have, by \eqref{r_tilde_derivative_at_b},
\begin{align} \label{the_condition2}
\eqref{the_condition}
 \Longleftrightarrow  e^{-\Phi(q)b}  \beta =1  +\delta \Phi(q) \int_0^b  e^{-\Phi(q) u}\mathbb{W}^{(q)}(u) \diff u
\Longleftrightarrow f(b) = 0.
\end{align}

Because  
\begin{align} \label{property_f}
f(0)   =1  -  \beta < 0, \quad f'(b)   = \Phi(q)   e^{-\Phi(q) b} \big( \delta \mathbb{W}^{(q)}(b) +  \beta \big) > 0, \quad b > 0,
\end{align}
and $f'(b) \xrightarrow{b \uparrow \infty} \infty$ by \eqref{W_q_limit},
 there must be a unique $b^* > 0$ such that $f(b^*) = 0$ or equivalently \eqref{the_condition} holds.

With this $b^*$, we have 
$
   (\delta  \mathbb{W}^{(q)}(b^*) + \beta) / \hat{r}^{(q) \prime}(b^*) = e^{\Phi(q)b^*} / \Phi(q)$.
Hence the corresponding candidate value function becomes, by \eqref{v_b_x},  for $x \geq 0$,
\begin{align} \label{value_function}
\begin{split}
v^{b^*}(x) 
&=  \delta  \frac {\mathbb{Z}^{(q)}(b^*-x)} q - \frac {e^{\Phi(q) b^*}} {\Phi(q)}  \hat{r}^{(q)}(b^*-x) \\
&=  \delta  \frac {\mathbb{Z}^{(q)}(b^*-x)} q -   \Big( \frac {e^{-\Phi(q)(x-b^*)}} {\Phi(q)}+\delta \int_x^{b^*}e^{-\Phi(q)(u-b^*)}\mathbb{W}^{(q)}(u-x)\diff u \Big).
\end{split}
\end{align}
In summary, we have the following.
\begin{lemma}  \label{smooth_fit}There exists a unique $b^* > 0$ such that $f(b^*) = 0$.  In addition $v^{b^*}$ is continuously differentiable at $b^*$, and in particular  for the case of unbounded variation it is twice continuously differentiable.  Namely, $v^{b^*}$ is sufficiently smooth on $(0, \infty)$.
\end{lemma}
%
%

\begin{remark} Recall that $b^*$ is strictly positive and, as is clear from \eqref{property_f}, $b^* \downarrow 0$ as $\beta \downarrow 1$. This is consistent with the results in \cite{BKY} without the absolutely continuous assumption. On the other hand, as in Avram et al.\ \cite{APP2007}, this is not necessarily the case in the  spectrally negative \lev model in particular when $Y$ is a compound Poisson process (with a drift) where $0$ is irregular for $(-\infty, 0)$.
\end{remark}
\begin{remark} \label{remark_convergence_delta}
(i) Because, as $\delta \rightarrow \infty$, $\Phi(q) \rightarrow 0$ and $\delta \Phi(q) = q - \psi_Y (\Phi(q)) \rightarrow q $ (see \eqref{def_varphi}), the function $f(b)$ as in \eqref{f_b} converges as $\delta \rightarrow \infty$ to $\mathbb{Z}^{(q)} (b) - \beta$, for all $b \geq 0$.
This is consistent with \cite{BKY} without the absolutely continuous assumption where the optimal reflection  barrier is $ (\mathbb{Z}^{(q)})^{-1} (\beta)$. \\ 
(ii) Furthermore, notice that \eqref{value_function} can be rewritten
\begin{align} 
\begin{split}
v^{b^*}(x) 
&=  \frac 1 q \Big(\delta - \frac q {\Phi(q)} \Big)   + \frac {1-e^{-\Phi(q)(x-b^*)}} {\Phi(q)}-\delta \int_x^{b^*} (e^{-\Phi(q)(u-b^*)} - 1) \mathbb{W}^{(q)}(u-x)\diff u. \label{v_b_star_rewrite}
\end{split}
\end{align}
With $b^*$ fixed constant, it converges to the value function of  \cite{BKY}:  $- \int_{0}^{b^*-x} \mathbb{Z}^{(q)} (y) \diff y - \psi_{Y,+}'(0)  / q$.  Indeed,
\begin{align*}
\delta - \frac q {\Phi(q)} = \delta - \frac {\psi_Y (\Phi(q)) + \delta \Phi(q)} {\Phi(q)} = - \frac {\psi_Y (\Phi(q))} {\Phi(q)} \xrightarrow{\delta \uparrow \infty} - \psi_{Y,+}'(0) \quad \textrm{and} \quad
\frac {1-e^{-\Phi(q)(x-b^*)}} {\Phi(q)} \xrightarrow{\delta \uparrow \infty} x - b^*.
\end{align*}
For the last integral of \eqref{v_b_star_rewrite},
\begin{align*}
\delta (e^{-\Phi(q)(u-b^*)} - 1) =  \delta \Phi(q) \frac {e^{-\Phi(q)(u-b^*)} - 1} {\Phi(q)} \xrightarrow{\delta \uparrow \infty} q (b^*-u),
\end{align*}
and hence
\begin{align*}
\delta \int_x^{b^*} (e^{-\Phi(q)(u-b^*)} - 1) \mathbb{W}^{(q)}(u-x)\diff u \xrightarrow{\delta \uparrow \infty} q \int_x^{b^*}(b^*-u) \mathbb{W}^{(q)}(u-x)\diff u = q  \int_0^{b^*-x} \overline{\mathbb{W}}^{(q)} (u) \diff u
\end{align*}
where the last equality holds by integration by parts.
Summing up these limits, the convergence is confirmed.
\end{remark}

\subsection{Verification} 

 We now show the optimality of the selected refraction-reflection strategy.
\begin{theorem}  \label{theorem_dividend}The refraction-reflection strategy with the refraction level $b^*$ and the reflection level $0$ is optimal and the value function is given by $v(x) = v^{b^*}(x)$ for all $0 \leq x < \infty$.
\end{theorem}

In order to show Theorem \ref{theorem_dividend}, we shall provide the following verification lemma and show that $v^{b^*}$ satisfies the stated conditions.

Following Proposition 4 (ii) in \cite{APP2007}, we extend the domain of the function $v_\pi$, for all $\pi \in \mathcal{A}$, as in \eqref{v_pi}, to all $\mathbb{R}$ by setting $v_\pi(x):=v_\pi(0)+\beta x$ for $x < 0$.
We let $\mathcal{L}_Y$ and $\mathcal{L}_X$ be the operators acting on sufficiently smooth functions $g$, defined by
\begin{equation*}
\begin{split}
\mathcal{L}_Y g(x)&:= -\gamma_Y g'(x)+\frac{\sigma^2}{2}g''(x) +\int_{(0,\infty)}[g(x + z)-g(x)-g'(x)z\mathbf{1}_{\{0<z<1\}}]\Pi(\mathrm{d}z), \\
\mathcal{L}_X g(x)&:= -(\gamma_Y+ \delta) g'(x)+\frac{\sigma^2}{2}g''(x) +\int_{(0,\infty)}[g(x+z)-g(x) -g'(x)z\mathbf{1}_{\{0<z<1\}}]\Pi(\mathrm{d}z).
\end{split}
\end{equation*}

\begin{lemma}[Verification lemma]
\label{verificationlemma}
Suppose $\hat{\pi}$ is an admissible dividend strategy such that $v_{\hat{\pi}}$ is sufficiently smooth on $(0,\infty)$, continuously differentiable 
at zero, and satisfies
\begin{align}
\label{HJB-inequality}
\sup_{0\leq r\leq\delta} \big((\mathcal{L}_Y - q)v_{\hat{\pi}}(x)-rv'_{\hat{\pi}}(x)+r \big) &\leq 0, \quad x > 0, \notag\\
v'_{\hat{\pi}}(x)&\leq\beta, \quad x > 0, \\
\inf_{x\geq 0}v_{\pi}(x) &> -m,\qquad\text{for some $m>0$.}\notag
\end{align} 
Then $v_{\hat{\pi}}(x)=v(x)$ for all $x\geq0$ and hence $\hat{\pi}$ is an optimal strategy.
\end{lemma}
\begin{proof}
See Appendix \ref{section_proof}.
\end{proof}

In the rest, we shall show that our candidate value function $v^{b^*}$ satisfies the sufficient conditions \eqref{HJB-inequality}.

\begin{lemma}\label{ver_1} (i) For $x > b^*$, we have $(\mathcal{L}_X - q)v^{b^*}(x)  + \delta = 0$.

 (ii) For $x < b^*$, we have $(\mathcal{L}_Y - q)v^{b^*}(x)  = 0$.
\end{lemma}
\begin{proof}
(i) For $x > b^*$, we have 
\begin{align}
 v^{b^*}(x)  &=  \frac \delta q -    \frac {e^{-\Phi(q)(x-b^*)}} {\Phi(q)}. \label{v_b_star_above}
\end{align}
Hence direct computation shows (i).

(ii) We know that $(\mathcal{L}_Y -q) \mathbb{Z}^{(q)}(b^*-x) = 0$ (see for instance the proof of Theorem 2.1 in \cite{BKY}).  In addition,
\begin{align*}
(\mathcal{L}_Y -q)e^{-\Phi(q)(x-b^*)} =(\mathcal{L}_X -q) (e^{-\Phi(q)(x-b^*)}) + \delta \Phi(q) e^{-\Phi(q)(x-b^*)} =\delta \Phi(q) e^{-\Phi(q)(x-b^*)}.
\end{align*}
Finally, by the proof of Lemma 4.5 of \cite{Yamazaki2013},
\begin{align*}
(\mathcal{L}_Y -q) \Big( \int_x^{b^*}e^{-\Phi(q)(u-b^*)}\mathbb{W}^{(q)}(u-x)\diff u  \Big) = - e^{-\Phi(q)(x-b^*)}.
\end{align*}
Summing up these, we have the claim.
\end{proof}

\begin{remark}\label{rem_ver} The function $v^{b^*}$ is continuously differentiable at zero. To see this, by  the choice of $b^*$ that satisfies \eqref{the_condition2}, the following holds:
\begin{align*}
v^{b^* \prime}_+(0)
&= e^{\Phi(q) b^*}\Big(  1  + \delta \Phi(q) \int_0^{b^*}e^{-\Phi(q)z}\mathbb{W}^{(q) }(z)\diff z \Big) = \beta = v^{b^* \prime}_-(0).
\end{align*}
\end{remark}

\begin{lemma}\label{ver_2} 
\begin{itemize}
\item[(i)] For $x \geq b^*$, we have $v^{b^* \prime}(x)  \leq 1 < \beta$. 
 \item[(ii)] For $0 \leq x \leq b^*$, we have $1 \leq v^{b^* \prime}(x)  \leq \beta$. 
 \end{itemize}
\end{lemma}
\begin{proof}
(i) For $x \geq b^*$, we have $v^{b^* \prime}(x)  = e^{-\Phi(q)(x-b^*)}  \leq 1$.

(ii)   For $0 \leq x \leq b^*$, by \eqref{r_tilde_derivative} and \eqref{value_function},
\begin{align*}
v^{b^* \prime}(x)
&= - \delta   {\mathbb{W}^{(q)}(b^*-x)}  + {e^{\Phi(q) b^*}}\Big(   e^{-\Phi(q)x} +\delta\Big(e^{-\Phi(q) b^*} \mathbb{W}^{(q)} (b^*-x) + \Phi(q) \int_x^{b^*} e^{-\Phi(q)u}\mathbb{W}^{(q)}(u-x)\diff u \Big) \Big) \\
&= e^{- \Phi(q) (x-b^*)}\Big(  1  + \delta \Phi(q) \int_0^{b^*-x}e^{-\Phi(q)z}\mathbb{W}^{(q) }(z)\diff z \Big).
\end{align*}
Taking its derivative, for $0 < x < b^*$,
\begin{align*}
v^{b^* \prime \prime}(x) &=  - \Phi(q) e^{- \Phi(q) (x-b^*)}\Big(  1  + \delta \Phi(q) \int_0^{b^*-x}e^{-\Phi(q)z}\mathbb{W}^{(q) }(z)\diff z \Big) \\
&- e^{- \Phi(q) (x-b^*)}\Big(  \delta \Phi(q) e^{-\Phi(q)(b^*-x)}\mathbb{W}^{(q) }(b^*-x) \Big) < 0.
\end{align*}
This together with Lemma \ref{smooth_fit} and Remark \ref{rem_ver}  shows the claim.\\
\end{proof}
Using the previous lemma, it is straightforward to check that the function $v^{b^*}$ is bounded from below. More specifically we have the following result. 
\begin{remark} \label{lemma_bound}
The function $v^{b^*}$ is bounded from below on $[0, \infty)$: by Lemma \ref{ver_2}  (ii) and \eqref{v_b_star_above},
\begin{align*}
\inf_{x\geq0}v^{b^*}(x) \geq \inf_{x \geq b^*}v^{b^*}(x)  \wedge v^{b^*}(0)  \geq \Big( \frac \delta q - \frac 1 {\Phi(q)} \Big)\wedge v^{b^*}(0)> -\infty.
\end{align*}
\end{remark}

Using the above lemmas, we now confirm that $v^{b^*}$ satisfies \eqref{HJB-inequality}. First, proceeding like in the proof of Lemma 7 in \cite{KLP},  Lemmas \ref{ver_1} and \ref{ver_2} imply the first item of \eqref{HJB-inequality}.  The second item of \eqref{HJB-inequality} is immediate by Lemma \ref{ver_2}. Finally, the third item is shown by Remark \ref{lemma_bound}. This completes the proof of Theorem \ref{theorem_dividend}.

\section{Numerical Examples} \label{numerical_section}

We conclude the paper with numerical examples of the optimal dividend problem studied above. Here, 
we focus on the case the processes $Y$ (and $X$) have i.i.d.\ phase-type distributed jumps, where their scale functions have analytical expressions, and hence the optimal strategy and the value function can be computed instantaneously.  The class of processes of this type is important because it can approximate any spectrally one-sided \lev process (see \cite{Asmussen_2004} and \cite{Egami_Yamazaki_2010_2}).  

Let $Y$ be a spectrally positive \lev process with i.i.d.\ phase-type distributed jumps \cite{Asmussen_2004} of the form
\begin{equation}
 Y_t  - Y_0= - c_Y t+\sigma B_t + \sum_{n=1}^{N_t} Z_n, \quad 0\le t <\infty, \label{X_phase_type}
\end{equation}
for some $c_Y \in \R$ and $\sigma \geq 0$.  Here $B=( B_t; t\ge 0)$ is a standard Brownian motion, $N=(N_t; t\ge 0 )$ is a Poisson process with arrival rate $\kappa$, and  $Z = ( Z_n; n = 1,2,\ldots )$ is an i.i.d.\ sequence of phase-type-distributed random variables with representation $(m,{\bm \alpha},{\bm T})$; see \cite{Asmussen_2004}.
These processes are assumed mutually independent. The Laplace exponents \eqref{lk} of $Y$ and $X$ are then (with ${\bm t} = -\bm{T 1}$ where ${\bm 1} = [1, \ldots 1]'$)
\begin{align*}
 \psi_Y(s)   &= c_Y s + \frac 1 2 \sigma^2 s^2 + \kappa \left( {\bm \alpha} (s {\bm I} - {\bm{T}})^{-1} {\bm t} -1 \right), \\
 \psi_X(s)  &= c_X s + \frac 1 2 \sigma^2 s^2 + \kappa \left( {\bm \alpha} (s {\bm I} - {\bm{T}})^{-1} {\bm t} -1 \right),
 \end{align*}
which are analytic for every $s \in \mathbb{C}$ except at the eigenvalues of ${\bm T}$.  

Suppose  $( -\zeta_{i,q}; i \in \mathcal{I}_q )$ and $( -\xi_{i,q}; i \in \mathcal{I}_q )$ are the sets of the roots with negative real parts of the equality $\psi_Y(s) = q$ and $\psi_X(s) = q$, respectively.  We assume that the phase-type distribution is minimally represented and hence $|\mathcal{I}_q| = m+1$ when $\sigma > 0$ and $|\mathcal{I}_q| = m$ when $\sigma = 0$; see \cite{Asmussen_2004}.  As in \cite{Egami_Yamazaki_2010_2}, if these values are assumed distinct, then
the scale functions of $-Y$ and $-X$ can be written, for all $x \geq 0$,
\begin{align}
\mathbb{W}^{(q)}(x) = \frac {e^{\varphi(q) x}} {\psi_Y'(\varphi(q))} - \sum_{i \in \mathcal{I}_q} B_{i,q} e^{-\zeta_{i,q}x} \quad \textrm{and} \quad
W^{(q)}(x) = \frac {e^{\Phi(q) x}} {\psi_X'(\Phi(q))} - \sum_{i \in \mathcal{I}_q} C_{i,q} e^{-\xi_{i,q}x},
\label{scale_function_form_phase_type}
\end{align}
respectively, where
\begin{align*}
B_{i,q} := \left. \frac { s+\zeta_{i,q}} {q-\psi_Y(s)} \right|_{s = -\zeta_{i,q}} = - \frac 1 {\psi_Y'(-\zeta_{i,q})} \quad \textrm{and} \quad
C_{i,q} := \left. \frac { s+\xi_{i,q}} {q-\psi_X(s)} \right|_{s = -\xi_{i,q}} = - \frac 1 {\psi_X'(-\xi_{i,q})}.
\end{align*}


\subsection{Numerical results}  We now confirm the optimality of the refraction-reflection strategy as obtained in the last section.  In the following numerical results, unless stated otherwise, we set $q = 0.05$, $\beta = 2$, and $\delta = 1$, and, for the \lev process, $c_Y = 0.5$ (and hence $c_X = 1.5$), $\sigma = 0.2$, $\kappa = 1$,  and  the jump size phase-type distribution given by $m=6$ and 
\begin{align*}
&{\bm T} = \left[ \begin{array}{rrrrrr}-4.0488  &  0.0000   & 0.0000  &  0.0000 &   0.0000 &   0.0000 \\
    0.1320  & -4.0012 &  0.0000  &  0.0455 &   3.7040  & 0.0044 \\
    0.2367  &  0.8595   &-4.2831  &  0.1897   & 0.2918   & 2.3724 \\
    3.1532   & 0.0000   & 0.0000 &  -4.0229  &  0.0000  &  0.0000 \\
    0.2497  &  0.0000  &  0.0000  &  3.7024  & -4.0124  &  0.0000 \\
    0.0434   & 2.1947  &  0.0938  &  0.1704  &  0.1217 &  -4.9612 \end{array} \right],  \quad {\bm \alpha} =    \left[ \begin{array}{r}   0.0052 \\
    0.0659 \\
    0.7446 \\
    0.0398 \\
    0.0043 \\
    0.1403  \end{array} \right],
\end{align*}
which gives an approximation to the (absolute value of) normal random variable with mean $0$ and variance $1$; we refer the reader to  \cite{Egami_Yamazaki_2010_2} for the accuracy of the approximation.

\begin{figure}[htbp]
\begin{center}
\begin{minipage}{1.0\textwidth}
\centering
\begin{tabular}{cc}
 \includegraphics[scale=0.42]{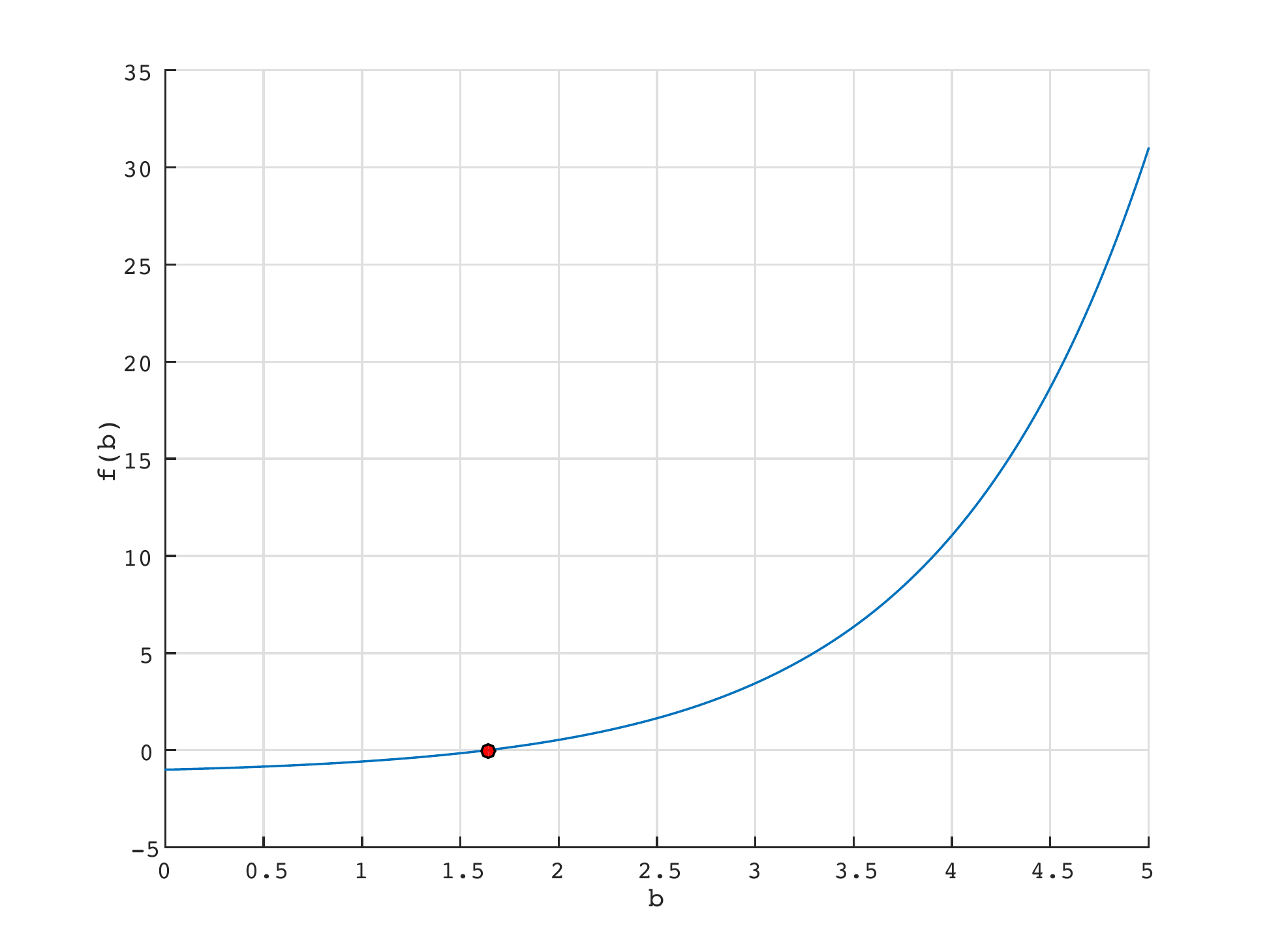} & \includegraphics[scale=0.42]{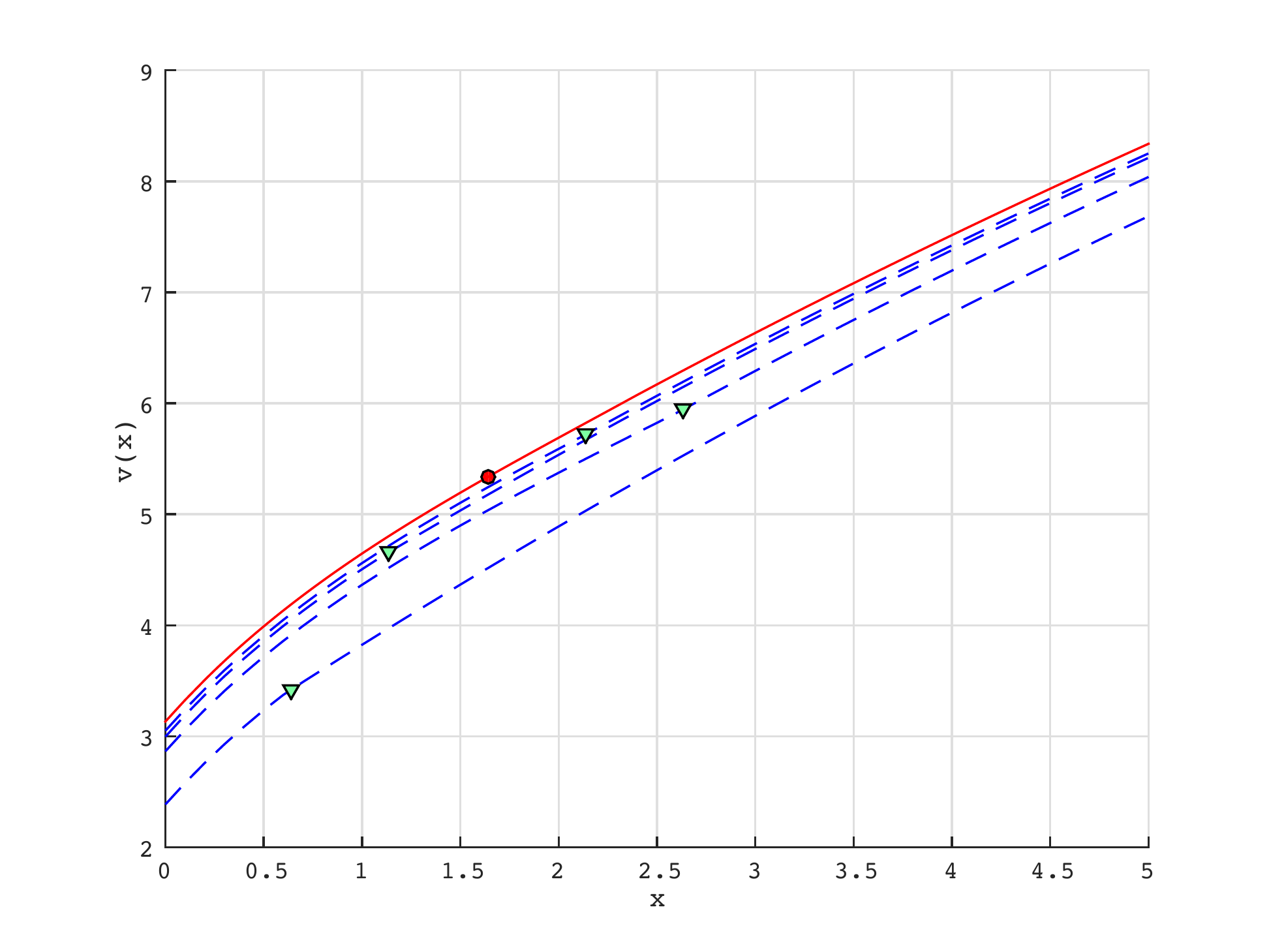}  \end{tabular}
\end{minipage}
\caption{Plots of $f(x)$ as in \eqref{f_b} (left) and the value function $v(x) = v^{b^*}(x)$ as in \eqref{value_function} (right).  For the former, the unique root $b^*$ is indicated by the circle.  For the latter, the value function $v(x)$ (solid) is shown along with $v^{b}(x)$ as in \eqref{v_b_x} (dotted) for $b \in \{b^*-1, b^*-0.5,b^*+0.5,b^*, b^*+1\}$.  The red circle shows the point $(b^*, v(b^*))$ and down-pointing triangles show the points $(b,v^{b}(b))$.
} \label{figure_caption_optimality}
\end{center}
\end{figure}

We first illustrate the computation of the optimal threshold level $b^*$ and the associated value function.  The function $f$ as in \eqref{f_b} is plotted in the left panel of Figure \ref{figure_caption_optimality}. As has been discussed, this function is confirmed to start at a negative value and increases strictly to infinity; hence its unique zero becomes the optimal refraction level $b^*$.   The right panel of Figure \ref{figure_caption_optimality} shows the value function \eqref{value_function} along with suboptimal NPVs $v^{b}(x)$, as in \eqref{v_b_x}, for $b \in \{b^*-1, b^*-0.5, b^*+0.5,b^*, b^*+1\}$. We can confirm that it dominates  the suboptimal NPVs uniformly in $x$.  In addition, as shown in Remark \ref{rem_ver} and Lemma \ref{ver_2}, it is a smooth concave function, whose slopes at $0$ and $b^*$ are $\beta$ and $1$, respectively.  

\begin{figure}[htbp]
\begin{center}
\begin{minipage}{1.0\textwidth}
\centering
\begin{tabular}{cc}
 \includegraphics[scale=0.42]{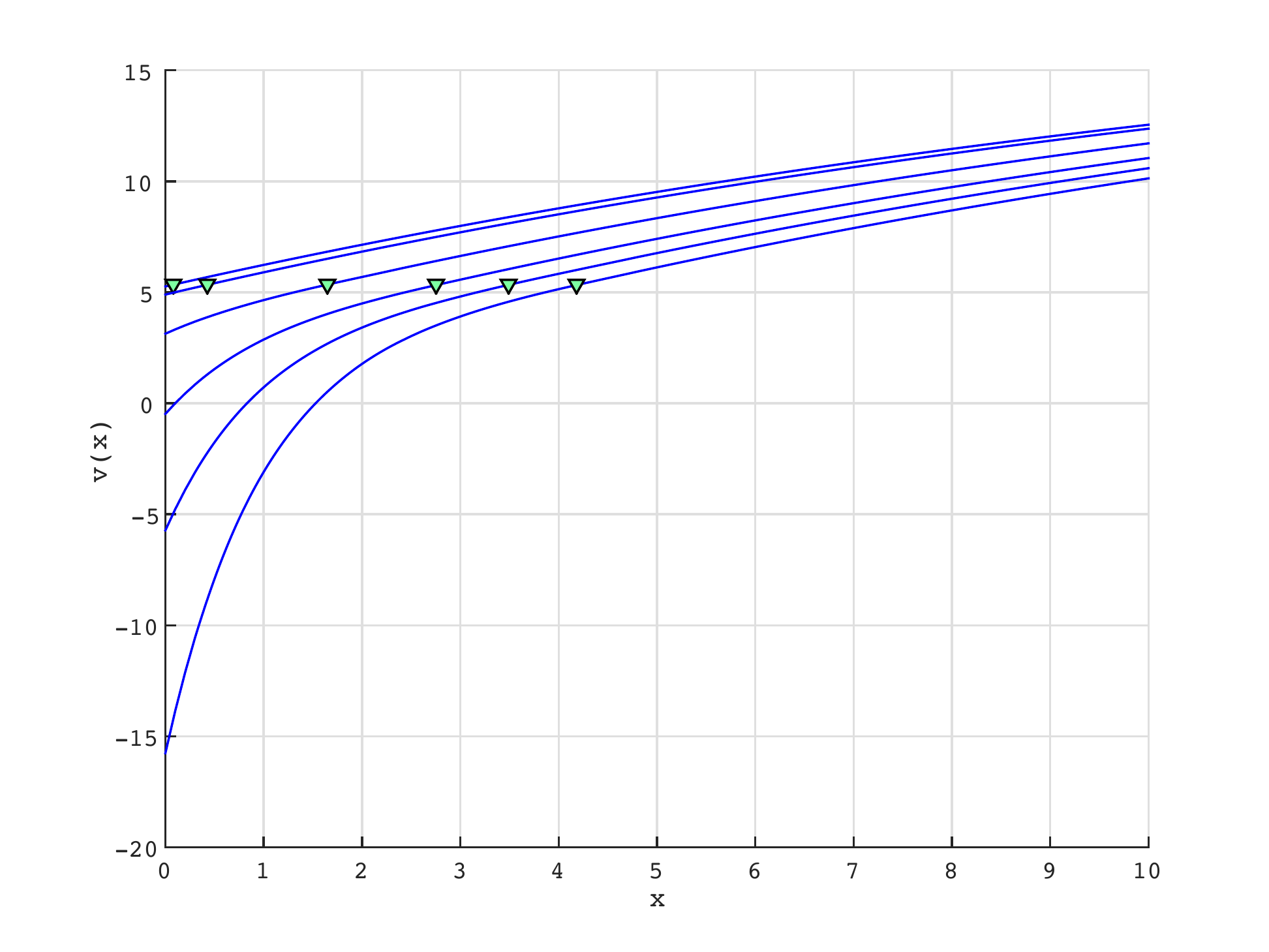} & \includegraphics[scale=0.42]{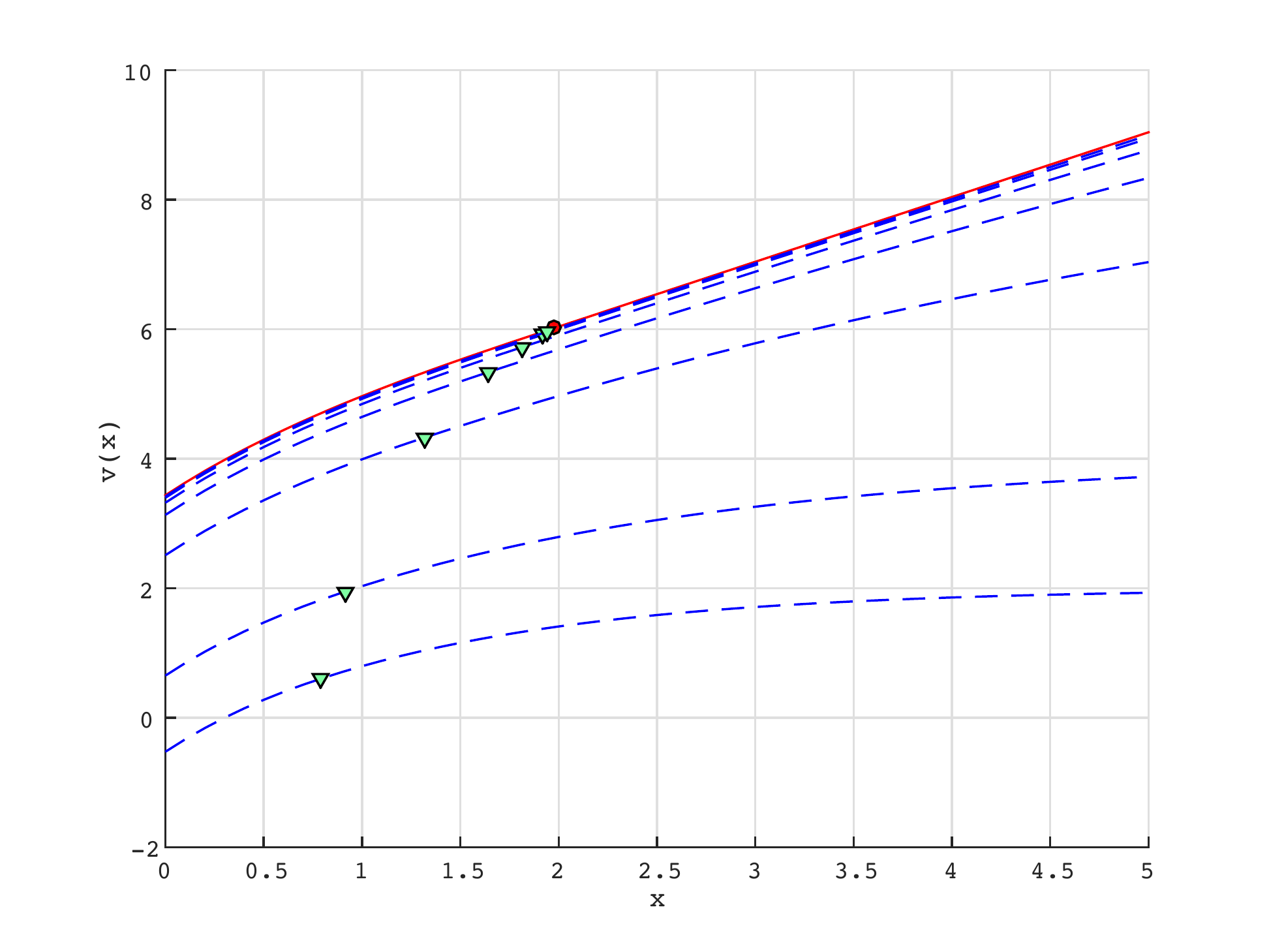}  \end{tabular}
\end{minipage}
\caption{Sensitivity of the value function $v(x)$ with respect to $\beta$ (left) and with respect to $\delta$ (right). 
The former is plotted for $\beta = 1.01,1.1,2,5,10, 20$ and the latter is for $\delta = 0.1, 0.2, 0.5, 1.0, 2, 5, 10$.  The points $(b^*, v(b^*))$ are indicated by down-pointing triangles. In particular, on the right panel, we plot the value functions (dotted) along with the value function without the absolutely continuous assumption as in \cite{BKY} (solid) with its optimal reflection level (red circle).
} \label{figure_sensitivity}
\end{center}
\end{figure}

We now study the sensitivity of the value function with respect to $\beta$ and $\delta$. The left panel of Figure \ref{figure_sensitivity} shows the value functions for various values of $\beta$: the value function is  decreasing in $\beta$ uniformly in $x$, and $b^*$ decreases to $0$ as $\beta \downarrow 1$.  The right panel shows those for various values of $\delta$ along with the unrestricted case \cite{BKY}: the value function is confirmed to be increasing in $\delta$ uniformly in $x$ and, as shown in Remark \ref{remark_convergence_delta}, the optimal refraction level $b^*$ as well as the value function converge to those in \cite{BKY}.


\appendix

\section{Proof of Lemma 4.2} \label{section_proof}
	
	By the definition of $v$ as a supremum, it follows that $v_{\hat{\pi}}(x)\leq v(x)$ for all $x\geq0$. We write $w:=v_{\hat{\pi}}$ and show that $w(x)\geq v_\pi(x)$ for all $\pi\in\mathcal{A}$ for all $x\geq0$. 
	
	Fix $\pi\in \mathcal{A}$. Here, without loss of generality, we can focus on those strategies $\pi$ such that $\int_{[0, \infty)}e^{-qs}\mathrm{d} R^\pi_s < \infty$ a.s. (which implies $R^\pi_t < \infty$ for all $t > 0$); otherwise, the fact that $\E_x [\int_{0}^{\infty}e^{-qs} \ell^\pi_s \mathrm{d}s ] \leq \delta /q$, implies $v_\pi (x) =-\infty$. 
	Let $(T_n)_{n\in\mathbb{N}}$ be the sequence of stopping times defined by $T_n :=\inf\{t>0:{V}^\pi_t>n \}$. 
	Since ${V}^\pi$ is a semi-martingale and $w$ is sufficiently smooth on $(0, \infty)$ and continuously differentiable 
	at zero by assumption, we can use the change of variables/Meyer-It\^o's formula (cf.\ Theorems II.31 and II.32 of \cite{protter})  to the stopped process $(e^{-q(t\wedge T_n)}w({V}^\pi_{t\wedge T_n}); t \geq 0)$ to deduce under $\mathbb{P}_x$ that
	\begin{equation*}
		\label{impulse_verif_1}
		\begin{split}
			e^{-q(t\wedge T_n)}w({V}^\pi_{t\wedge T_n})-w(x)
			= & -\int_{0}^{t\wedge T_n}e^{-qs} q w({V}^\pi_{s-}) \mathrm{d}s
			+\int_{[0, t\wedge T_n]}e^{-qs}w'({V}^\pi_{s-}) \mathrm{d}  ( Y_s- {L}^\pi_s  )  \\
			&
			+ \frac{\sigma^2}{2}\int_0^{t\wedge T_n}e^{-qs}w''({V}^\pi_{s-})\mathrm{d}s \\
			&+\int_{[0, t\wedge T_n]}e^{-qs}w'({V}^\pi_{s-}) \mathrm{d} R_s^{\pi,c} + \sum_{0 \leq s\leq t\wedge T_n}e^{-qs}[\Delta w({V}^\pi_{s-}+\Delta R_s^\pi)]\\
			& + \sum_{0 \leq s\leq t\wedge T_n}e^{-qs}[\Delta w({V}^\pi_{s-}+\Delta Y_s)-w'({V}^\pi_{s-})  \Delta Y_s  ],
		\end{split}
	\end{equation*}
	where $R^{\pi,c}$ is the continuous part of $R^{\pi}$ and we use the following notation: $\Delta \zeta_s:= \zeta_s-\zeta_{s-}$ and $\Delta w(\zeta_s):=w(\zeta_s)-w(\zeta_{s-})$ for any  process $\zeta$. 
	Rewriting the above equation leads to 
	\begin{equation*}
		\begin{split}
			e^{-q(t\wedge T_n)}w({V}^\pi_{t\wedge T_n})  -w(x)
			&=   \int_{0}^{t\wedge T_n}e^{-qs}   (\mathcal{L}_Y-q)w({V}^\pi_{s-})   \mathrm{d}s
			-\int_{0}^{t\wedge T_n}e^{-qs}w'({V}^\pi_{s-})\mathrm{d}{L}^\pi_s  
			+ M_{t \wedge T_n}\\
			&+\int_{[0, t\wedge T_n]}e^{-qs}w'({V}^\pi_{s-}) \mathrm{d} R_s^{\pi,c} + \sum_{0 \leq s\leq t\wedge T_n}e^{-qs}[\Delta w({V}^\pi_{s-}+\Delta R_s^\pi)],
			\end{split}
			\end{equation*}
			where 
			\begin{align}\label{def_M_martingale}
			\begin{split}
				M_t &:= \int_0^t \sigma  e^{-qs} w'(V_{s-}^{\pi}) \diff B_s +\lim_{\varepsilon\downarrow 0}\int_{[0,t]} \int_{(\varepsilon,1)}  e^{-qs}w'(V_{s-}^{\pi})y (N(\diff s\times \diff y)-\Pi(\diff y) \diff s)\\
				&+\int_{[0,t]} \int_{(0,\infty)} e^{-qs}(w(V_{s-}^\pi+y)-w(V_{s-}^\pi)-w'(V_{s-}^\pi)y\mathbf{1}_{\{y\in (0, 1)\}})(N(\diff s\times \diff y)-\Pi(\diff y) \diff s), \quad t \geq 0.
				\end{split}
			\end{align}
			Here, $( B_s; s \geq 0 )$ is a standard Brownian motion and $N$ is a Poisson random measure in   the measure space  $([0,\infty)\times (0, \infty),\B [0,\infty)\times \B (0, \infty), \diff s \times \Pi( \diff x))$.
		
					\par On the other hand using the fact that $w'(x)\leq \beta$ for $x>0$, we obtain that
					\begin{align*}
			&\int_{[0, t\wedge T_n]}e^{-qs}w'({V}^\pi_{s-}) \mathrm{d} R_s^{\pi,c} + \sum_{0 \leq s\leq t\wedge T_n}e^{-qs}[\Delta w({V}^\pi_{s-}+\Delta R_s^\pi)]\\
						&\leq \beta\int_{[0, t\wedge T_n]}e^{-qs} \mathrm{d} R_s^{\pi,c} +\beta\sum_{0 \leq s\leq t\wedge T_n} e^{-qs} \Delta R_s^{\pi}=\beta\int_{[0, t\wedge T_n]}e^{-qs} \mathrm{d} R^{\pi}_s.
					\end{align*}
	\par Hence we derive that
	\begin{equation*}
		\begin{split}
			w(x) \geq&
			-\int_{0}^{t\wedge T_n}e^{-qs}  \left[ (\mathcal{L}_Y-q)w({V}^\pi_{s-})-\ell^\pi_sw'({V}^\pi_{s-})+ \ell^\pi_s \right]  \mathrm{d}s-\beta\int_{[0, t\wedge T_n]}e^{-qs} \mathrm{d} R^{\pi}_s \\
			& + \int_{0}^{t\wedge T_n}e^{-qs} \ell^\pi_s \mathrm{d}s - M_{t\wedge T_n} + e^{-q(t\wedge T_n)}w({V}^\pi_{t\wedge T_n}).
		\end{split}
	\end{equation*}
	Using  the assumption \eqref{HJB-inequality} and that $V_{s-}^\pi \geq 0$ and $\ell_s^\pi \in [0, \delta]$ a.s.\ for all $s \geq 0$, we have
	\begin{equation} \label{w_lower}
		\begin{split}
			w(x) \geq &
			\int_{0}^{t\wedge T_n}e^{-qs} \ell^\pi_s \mathrm{d}s -\beta\int_{[0, t\wedge T_n]}e^{-qs}\mathrm{d} R^\pi_s - M_{t\wedge T_n}- m e^{-q(t\wedge T_n)}.
		\end{split}
	\end{equation} 
	In addition by the compensation formula (cf.\ Corollary 4.6 of \cite{K}), $(M_{t \wedge T_n}:t\geq0 )$ is a zero-mean $\mathbb{P}_x$-martingale.   
	\par Now  taking expectations  in \eqref{w_lower} and
	letting $t$ and $n$ go to infinity ($T_n\nearrow\infty$ $\mathbb{P}_x$-a.s.), the monotone convergence theorem gives  (applied separately for $\E_x [\int_{0}^{t\wedge T_n}e^{-qs} \ell^\pi_s \mathrm{d}s ]$ and $\E_x (\beta\int_{[0, t\wedge T_n]}e^{-qs}\mathrm{d} R^\pi_s) $)
	\begin{equation*}
		w(x) \geq \mathbb{E}_x \left( \int_{0}^{\infty}e^{-qs} \ell^\pi_s \mathrm{d}s-\beta\int_{[0,\infty)}e^{-qs}\mathrm{d} R^\pi_s \right) =v_\pi(x),
	\end{equation*}
	where the expectation makes sense because $\E_x (\int_{0}^{\infty}e^{-qs} \ell^\pi_s \mathrm{d}s) \in [0, \delta/q]$.
	This completes the proof.






\end{document}